\documentclass[a4paper,10pt]{article}

\pdfoutput=1 

\usepackage{geometry}
\usepackage{empheq}
\usepackage{amsmath,amsfonts,amssymb,amsthm,mathtools,tikz,enumitem}
\usepackage{caption}
\usepackage{graphicx}
\usepackage{float}
\usepackage{subcaption}
\usepackage{microtype}
\usepackage{url,hyperref}
\usepackage{indentfirst}
\hypersetup{colorlinks, linkcolor={red}, citecolor={blue}, urlcolor={black}}
\usepackage{cleveref}

\usepackage[T1]{fontenc}
\usetikzlibrary{decorations.pathreplacing,calc}

\allowdisplaybreaks

\newtheorem{theorem}{Theorem}[section]
\newtheorem{lemma}[theorem]{Lemma}
\newtheorem{corollary}[theorem]{Corollary}

\newtheorem{claim}[theorem]{Claim}
\newtheorem{problem}[theorem]{Problem}

\theoremstyle{definition}

\crefname{theorem}{Theorem}{Theorems}
\Crefname{theorem}{Theorem}{Theorems}
\crefname{lemma}{Lemma}{Lemmas}
\Crefname{lemma}{Lemma}{Lemmas}
\crefname{corollary}{Corollary}{Corollaries}
\Crefname{corollary}{Corollary}{Corollaries}
\crefname{claim}{Claim}{Claims}
\Crefname{claim}{Claim}{Claims}
\Crefname{problem}{Problem}{Problems}

\newcommand{\dd}{\mathrm{d}}
\newcommand{\GG}{\mathcal{G}}
\newcommand{\bA}{\mathbf{A}}

\begin{document}

\title{On clique-to-clique densities}
\date{}

\author{
Jie Ma\thanks{School of Mathematical Sciences, University of Science and Technology of China, Hefei, Anhui 230026, China.} \(^{,}\)\thanks{Yau Mathematical Sciences Center, Tsinghua University, Beijing 100084, China.}
\and
Tianhen Wang\footnotemark[1]
\and 
Tianming Zhu\footnotemark[1]
}

\maketitle

\begingroup
\renewcommand\thefootnote{}
\footnotetext{Emails:~~~\texttt{jiema@ustc.edu.cn};~~~\texttt{wth1115060377@mail.ustc.edu.cn};~~~\texttt{zhutianming@mail.ustc.edu.cn}.}
\endgroup

\begin{abstract}
Let \(k_r(G)\) denote the number of \(r\)-cliques in a graph \(G\) and let \(F_r(\cdot)\) be the Lov\'asz--Simonovits \(r\)-clique density function. For any integers \(2\le s<t\), we determine the asymptotically sharp lower bound on \(k_t(G)\) in an \(n\)-vertex graph \(G\) with a prescribed number \(k_s(G)\), by showing that
\[
\frac{k_t(G)}{n^t}\ge F_t\!\left(F_s^{-1}\!\left(\frac{k_s(G)}{n^s}\right)\right),
\]
where \(F_s^{-1}\) denotes the generalized inverse.
This strengthens Bollob\'as's piecewise-linear interpolation bound and, in the case \(s=2\), recovers Reiher's clique density theorem via a new inductive proof.
\end{abstract}

\section{Introduction}\label{SEC:Intro}
\noindent 
For an integer \(r \ge 2\) and a graph \(G\), let \(k_r(G)\) denote the number of copies of the \(r\)-clique \(K_r\) in \(G\). In this paper, we investigate extremal problems concerning the minimization of the \(r\)-clique density \(p_r(G) \coloneqq k_r(G) / |V(G)|^r\) over graphs \(G\) with prescribed density conditions.

A cornerstone of extremal graph theory is Tur\'an's theorem~\cite{turan1941extremal}, which states that for any integers \(n\ge r>2\), the maximum number \(k_2(G)\) of edges among all \(n\)-vertex \(K_r\)-free graphs \(G\) (i.e., containing no copy of \(K_r\)) is attained by the Tur\'an graph \(T_{r-1}(n)\), the balanced complete \((r-1)\)-partite graph on \(n\) vertices. Generalizations of this classical result have played a central role in the development of extremal graph theory. In particular, Zykov~\cite{zykov1949linear} and, independently, Erd\H{o}s~\cite{erdos1962number} proved that for any integers \(n\ge r>s\ge2\), the maximum number \(k_s(G)\) among all \(n\)-vertex \(K_r\)-free graphs \(G\) is likewise attained by \(T_{r-1}(n)\). This result initiated the study of generalized Tur\'an numbers (see, e.g., \cite{alon2016}).
These theorems imply that if \(k_2(G)>k_2(T_{r-1}(n))\), or more generally if \(k_s(G)>k_s(T_{r-1}(n))\), then \(G\) necessarily contains a copy of \(K_r\). 
This naturally leads to the corresponding supersaturation problem: given the number of copies of \(K_s\) in a graph \(G\), what is the minimum possible number of copies of \(K_r\)?

We refer to this as the \(K_s\to K_r\) problem. It has been studied extensively in the literature, with particular emphasis on the case \(s=2\).
For integers \(p, r\ge 2\), we define the \emph{critical point}
\begin{equation}\label{eq:theta-intro}
    \theta_{p,r}\coloneqq \frac{\binom{p+1}{r}}{(p+1)^r}
    =\lim_{n\to\infty}p_r\big(T_{p+1}(n)\big),
\end{equation}
which is the limiting \(r\)-clique density of the Tur\'an graphs \(T_{p+1}(n)\).
Following earlier results on triangles and \(K_4\) in~\cite{goodman1959sets,moon1962problem,nordhaus1963triangles},
Khad\v{z}iivanov and Nikiforov~\cite{Khadzhiivanov1978} proved that for any integer $r\ge 2$ and any graph \(G\) with edge-density \(\gamma\), if \(\gamma\ge \frac{r-2}{2(r-1)},\)
then the \(r\)-clique density of \(G\) satisfies
\(p_r(G)\ge \frac{1}{r!}\prod_{j=2}^r \bigl(2(j-1)\gamma-(j-2)\bigr).\)
At each critical point \(\gamma=\theta_{p,2}\) for \(p\ge r-2\), this bound becomes \(p_r(G)\ge \theta_{p,r}\) and is attained asymptotically by the corresponding Tur\'an graphs.
It is worth noting that the right-hand side is a convex function of \(\gamma\) on \([\frac{r-2}{2(r-1)},1/2)\).

In~\cite{bollobas1976complete}, Bollob\'as used an elegant optimization argument to obtain the following far-reaching strengthening. 
For any integers $r>s\ge 2$ and any graph \(G\) with \(p_s(G)=\gamma\), 
the \(r\)-clique density of \(G\) satisfies \(p_r(G)\geq L_{s,r}(\gamma)\), where \(L_{s,r}(\gamma)\colon [0,1/s!)\to[0,1/r!)\) denotes the unique piecewise-linear function satisfying
\(L_{s,r}(0)=0\) and \(L_{s,r}(\theta_{p,s})=\theta_{p,r}\)
for every integer \(p\ge r-2\).
Evidently, this bound is optimal whenever \(\gamma=\theta_{p,s}\) for integers \(p\ge r-2\), and in the case \(s=2\), it improves the Khad\v{z}iivanov--Nikiforov bound at every non-critical point.

In their seminal work~\cite{lovasz1983number}, Lov\'asz and Simonovits formulated the celebrated clique density conjecture for the \(K_2\to K_r\) problem. It asserts that, for every edge density \(\gamma\), the minimum possible \(r\)-clique density is asymptotically attained by a complete multipartite graph in which all but at most one vertex class have the same size, while the remaining class may be smaller.
More precisely, for each \(\gamma\in(0,1/2)\), let \(p\in\mathbb{N}_{\ge1}\) and \(\alpha\in[0,1/p)\) be the unique parameters satisfying
\(\gamma=\frac{p}{2(p+1)}(1-\alpha^2)\),
and let \(K_{\gamma,n}\) denote the \(n\)-vertex complete \((p+1)\)-partite graph with \(p\) vertex classes of size
\(\frac{(1+\alpha)n}{p+1}\) and one vertex class of size \(\frac{(1-p\alpha)n}{p+1}\).
For \(\gamma=0\), let \(K_{0,n}\) be the empty graph on \(n\) vertices. Define
\(F_r(\gamma)\coloneqq\lim_{n\to\infty}p_r(K_{\gamma,n})\) for \(\gamma\in[0,1/2)\).
The clique density conjecture of Lov\'asz--Simonovits~\cite{lovasz1983number} then states that for every \(r\ge 3\) and every graph \(G\), \(p_r(G)\ge F_r\big(p_2(G)\big)\).

Before discussing subsequent developments, we pause to examine the function \(F_r\colon [0,1/2)\to[0,1/r!)\).
Clearly, \(F_2(\gamma)=\gamma\). Moreover, for every \(r\ge3\), \(F_r(0)=0\), and a direct calculation yields
\begin{equation}\label{eq:F-def}
F_r(\gamma)
=\frac{\binom{p+1}{r}}{(p+1)^r}
(1+\alpha)^{r-1}\bigl(1-(r-1)\alpha\bigr)
\quad \mbox{for} \quad \gamma\in(0,1/2).
\end{equation}
The function \(F_r\) vanishes identically on \([0,\frac{r-2}{2(r-1)}]\),\footnote{Indeed, if \(p\le r-2\), then \(\binom{p+1}{r}=0\), and hence \(F_r(\gamma)=0\).} is continuous and strictly increasing on \([\frac{r-2}{2(r-1)},1/2)\), and is piecewise concave on each interval between consecutive critical points \(\theta_{p,2}\) and \(\theta_{p+1,2}\), where \(p\ge r-2\); see Figure~\ref{fig:Curve} (with \(s=2\)).
Finally, we define the inverse function
\(F_r^{-1}\colon [0,1/r!)\to[0,1/2)\)
by setting
\(F_r^{-1}(0)\coloneqq\frac{r-2}{2(r-1)}\),
and, for \(y\in(0,1/r!)\), letting \(F_r^{-1}(y)\) denote the inverse of \(F_r\) on its strictly increasing part.

Regarded as one of the central problems in extremal graph theory, the clique density conjecture has received considerable attention since its formulation.
Lov\'asz and Simonovits~\cite{lovasz1983number} proved it in a small right neighborhood of every critical point: for every fixed integer \(r\ge3\) and every integer \(p\ge r-2\), there exists \(\varepsilon_p>0\) such that the conjectured lower bound holds whenever
\(\theta_{p,2}\le p_2(G)\le\theta_{p,2}+\varepsilon_p\). 
The first non-trivial interval \(1/4\le p_2(G)\le 1/3\)
for the case \(r=3\) was established by Fisher~\cite{fisher1989lower} and later revisited by Razborov~\cite{razborov2007flag} as one of the first applications of the flag algebra method.
Shortly thereafter, Razborov~\cite{razborov2008minimal} completely resolved the case \(r=3\), again using the flag algebra method.
Nikiforov~\cite{nikiforov2011number} subsequently introduced a weighted graph analytic approach, which yielded a new proof for the case \(r=3\) and established the case \(r=4\). Finally, in a major breakthrough, Reiher~\cite{reiher2016clique} built on this weighted graph framework to prove the clique density conjecture for every \(r\ge3\).

\begin{theorem}[Reiher~\cite{reiher2016clique}, The clique density theorem]\label{THM:Reiher}
Let \(r\ge 3\) be an integer. Then for every graph \(G\),
\[
        p_r(G)\ge F_r(p_2(G)).
\]
\end{theorem}

In this paper, we determine the asymptotically sharp lower bound on \(k_t(G)\) for graphs \(G\) with prescribed \(k_s(G)\), thereby resolving the general \(K_s\to K_t\) problem for all integers \(t>s\ge2\). Our main result is as follows.

\begin{theorem}\label{THM:ordinary-main}
Let \(t>s\ge 2\) be integers. Then for every graph \(G\),
\[
        p_t(G)\ge F_t\bigl(F_s^{-1}(p_s(G))\bigr).
\]
\end{theorem}

This bound is sharp at every \(s\)-clique density, as witnessed by the complete multipartite graph \(K_{\gamma,n}\) introduced above. For integers \(t>s\ge2\), the composition \(F_t\circ F_s^{-1}\colon [0,1/s!)\to[0,1/t!)\) is piecewise concave on each interval between consecutive critical points \(\theta_{p,s}\) and \(\theta_{p+1,s}\), where \(p\ge t-2\); see Figure~\ref{fig:Curve}. As a consequence, it improves Bollob\'as's piecewise-linear interpolation bound~\cite{bollobas1976complete} at every non-critical point. Theorem~\ref{THM:ordinary-main} also yields the aforementioned theorem of Erd\H{o}s--Zykov on generalized Tur\'an numbers. In the case \(s=2\), it recovers Reiher's clique density theorem (Theorem~\ref{THM:Reiher}). Our proof is based on a new inductive argument that is conceptually simple; see Section~\ref{SEC:Overview} for an overview.

A direct consequence of Theorem~\ref{THM:ordinary-main} is a hierarchy among all clique densities for every graph. Note that for every integer \(r \ge 2\), the inverse function \(F_r^{-1}\) is strictly increasing, and hence preserves the ordering of clique-density parameters. Also note that \(F_2(\gamma)=\gamma\).

\begin{corollary}\label{Coro:density}
    For every graph \(G\), it holds that
    \[
    p_2(G)=F_2^{-1}\big(p_2(G)\big) \le F_3^{-1}\big(p_3(G)\big)\le F_4^{-1}\big(p_4(G)\big) \le \cdots \le F_r^{-1}\big(p_r(G)\big) \le F_{r+1}^{-1}\big(p_{r+1}(G)\big)\le \cdots.
    \]
    Moreover, each inequality is sharp.
\end{corollary}

 \begin{figure}[htbp]
 \begin{minipage}{0.49\linewidth}
      \centering
      \begin{tikzpicture}[
        xscale=0.64, yscale=0.7,
        >=stealth,
        point/.style={circle,draw=black, fill=white,inner sep=1.5pt},
        proj x/.style={densely dotted, gray, thin},
        proj y/.style={densely dotted, gray, thin},
        label x/.style={font=\footnotesize, below=2pt},
        label y/.style={font=\footnotesize, left=2pt},
    ]
    \draw[->, thick] (1.1,0) -- (11.0,0)
        node[right] {\(p_3\)};
    \draw[->, thick] (1.1,0) -- (1.1,5.0)
        node[above] {\(p_4\)};

    \node[below left] at (1.1,0) {\(0\)};

    \coordinate (P0) at (2.85, 0);       
    \coordinate (P1) at (4.500, 0.563);  
    \coordinate (P2) at (5.760, 1.152);   
    \coordinate (P3) at (6.667, 1.667);  
    \coordinate (P4) at (7.347, 2.099);  
    \coordinate (P5) at (7.785, 2.401);  
    \coordinate (P6) at (8.287, 2.765);   
    \coordinate (PE) at (10.000, 4.500);  
    \draw[thick, red] (1.1,0) -- (P0);

    \draw[thick, red]
        (P0) .. controls (3.5, 0.40) and (3.8, 0.52) .. (P1);

    \draw[thick, red]
        (P1) .. controls (4.95, 0.95) and (5.2, 1.07) .. (P2);

    \draw[thick, red]
        (P2) .. controls (6.1, 1.52) and (6.3, 1.62) .. (P3);

    \draw[thick, red]
        (P3) .. controls (6.9, 1.93) and (7.0, 2.0) .. (P4);

    \draw[thick, red]
        (P4) .. controls (7.5, 2.30) and (7.6, 2.36) .. (P5);

    \draw[thick, red]
        (P5) .. controls (7.9, 2.50) and (8.0, 2.7) .. (P6);

    \draw[thick, blue] (P0) -- (P1);
    \draw[thick, blue] (P1) -- (P2);
    \draw[thick, blue] (P2) -- (P3);
    \draw[thick, blue] (P3) -- (P4);
    \draw[thick, blue] (P4) -- (P5);
    \draw[thick, blue] (P5) -- (P6);

    \draw[thick, gray, dotted] (P6) to[out=36, in=-125] (PE);

    \draw[proj x] (P0) -- (P0 |- 0,0)
        node[label x] {\(\frac{1}{27}\)};

    \draw[proj x] (P1) -- (P1 |- 0,0)
        node[label x] {\(\frac{1}{16}\)};
    \draw[proj y] (P1) -- (P1 -| 1.1,0)
        node[label y] {\(1/256\)};

    \draw[proj x] (P2) -- (P2 |- 0,0)
        node[label x] {\(\frac{2}{25}\)};
    \draw[proj y] (P2) -- (P2 -| 1.1,0)
        node[label y] {\(1/125\)};

    \draw[proj x] (P3) -- (P3 |- 0,0)
        node[label x] {\(\frac{10}{108}\)};
    \draw[proj y] (P3) -- (P3 -| 1.1,0)
        node[label y] {\(15/1296\)};

    \draw[proj x] (P4) -- (P4 |- 0,0)
        node[label x] {\(\frac{35}{343}\)};
    \draw[proj y] (P4) -- (P4 -| 1.1,0)
        node[label y] {\(35/2401\)};

    \draw[proj x] (PE) -- (PE |- 0,0)
        node[label x] {\(\frac{1}{6}\)};
    \draw[proj y] (PE) -- (PE -| 1.1,0)
        node[label y] {\(1/24\)};

    \foreach \p in {P0,P1,P2,P3,P4,P5,P6,PE} {
      \node[point] at (\p) {};
    }

    \end{tikzpicture}
    \end{minipage}
    \begin{minipage}{0.49\linewidth}
      \centering
      \begin{tikzpicture}[
          xscale=0.64, yscale=0.7,
          >=stealth,
          point/.style={circle,draw=black, fill=white,inner sep=1.5pt},
          proj x/.style={densely dotted, gray, thin},
          proj y/.style={densely dotted, gray, thin},
          label x/.style={font=\footnotesize, below=2pt},
          label y/.style={font=\footnotesize, left=2pt},
      ]

      \draw[->, thick] (1.1,0) -- (11.0,0)
          node[right] {\(p_s\)};
      \draw[->, thick] (1.1,0) -- (1.1,5.0)
          node[above] {\(p_t\)};

      \node[below left] at (1.1,0) {\(0\)};

      \coordinate (P0) at (2.85, 0);   
      \coordinate (P1) at (4.500, 0.563); 
      \coordinate (P2) at (5.760, 1.152); 
      \coordinate (P3) at (6.667, 1.667); 
      \coordinate (P4) at (7.347, 2.099);
      \coordinate (P5) at (7.785, 2.401); 
      \coordinate (P6) at (8.287, 2.765); 
      \coordinate (PE) at (10.000, 4.500); 

      \draw[thick, red] (1.1,0) -- (P0);

      \draw[thick, red]
          (P0) .. controls (3.5, 0.40) and (3.8, 0.52) .. (P1);

      \draw[thick, red]
          (P1) .. controls (4.95, 0.95) and (5.2, 1.07) .. (P2);

      \draw[thick, red]
          (P2) .. controls (6.1, 1.52) and (6.3, 1.62) .. (P3);

      \draw[thick, red]
          (P3) .. controls (6.9, 1.93) and (7.0, 2.0) .. (P4);

      \draw[thick, red]
          (P4) .. controls (7.5, 2.30) and (7.6, 2.36) .. (P5);

      \draw[thick, red]
          (P5) .. controls (7.9, 2.50) and (8.0, 2.7) .. (P6);

      \draw[thick, gray, dotted] (P6) to[out=36, in=-125] (PE);

      \draw[proj x] (P0) -- (P0 |- 0,0)
          node[label x] {\(\theta_{t-2,s}\)};

      \draw[proj x] (P1) -- (P1 |- 0,0)
          node[label x] {\(\theta_{t-1,s}\)};
      \draw[proj y] (P1) -- (P1 -| 1.1,0)
          node[label y] {\(\theta_{t-1,t}\)};

      \draw[proj x] (P2) -- (P2 |- 0,0)
          node[label x] {\(\theta_{t,s}\ \)};
      \draw[proj y] (P2) -- (P2 -| 1.1,0)
          node[label y] {\(\theta_{t,t}\)};

      \draw[proj x] (P3) -- (P3 |- 0,0)
          node[label x] {\(\ \theta_{t+1,s}\)};
      \draw[proj y] (P3) -- (P3 -| 1.1,0)
          node[label y] {\(\theta_{t+1,t}\)};

      \draw[proj x] (PE) -- (PE |- 0,0)
          node[label x] {\(\frac{1}{s!}\)};
      \draw[proj y] (PE) -- (PE -| 1.1,0)
          node[label y] {\(\frac{1}{t!}\)};

      \foreach \p in {P0,P1,P2,P3,P4,P5,P6,PE} {
        \node[point] at (\p) {};
      }

  \end{tikzpicture}
    \end{minipage}
    \caption{}
      \caption*{%
\textbf{Left:} The red curve shows the bound of Theorem~\ref{THM:ordinary-main} for \((s,t)=(3,4)\), while the blue curve shows Bollob\'as's corresponding piecewise-linear interpolation bound. 
\textbf{Right:} A schematic illustration of the curve given by Theorem~\ref{THM:ordinary-main} for general \(t>s\ge 2\), where the quantities \(\theta_{p,r}\) are defined in~\eqref{eq:theta-intro}.
In both figures, the dotted extensions indicate that infinitely many further pieces accumulate at the point \((1/s!,\,1/t!)\).
}
\label{fig:Curve}
\end{figure}

The rest of the paper is organized as follows. In Section~\ref{SEC:WeightedGraphs}, we introduce weighted graphs and reduce the proof of Theorem~\ref{THM:ordinary-main} to the clique lifting theorem, Theorem~\ref{THM:clique-lifting}. In Section~\ref{SEC:Overview}, we outline the proof of Theorem~\ref{THM:clique-lifting}. Section~\ref{SEC:Analytic} collects several auxiliary lemmas, including one on analytic function extension. In Section~\ref{SEC:MainProof}, we prove Theorem~\ref{THM:clique-lifting} by induction on \(r\). In Section~\ref{SEC:concluding}, we establish a stability result for Theorem~\ref{THM:ordinary-main} and discuss a related open problem for future work.
Throughout the paper, we write \([n]=\{1,2,\dots,n\}\) for a positive integer \(n\), and for a set \(X\) we denote by \(\binom{X}{r}\) the family of all \(r\)-element subsets of \(X\).

\section{Weighted Graphs and Clique Lifting}\label{SEC:WeightedGraphs}
\noindent
In this section, we formulate our main result in the setting of weighted graphs and reduce the proof of Theorem~\ref{THM:ordinary-main} to a weighted graph statement (Theorem~\ref{THM:clique-lifting}).
A {\it weighted graph} \(\GG(\vec{\mathbf{x}}, \bA)\) of order \(n\) consists of a vector \(\vec{\mathbf{x}} = (x_1, x_2, \dots, x_n)\in \mathbb{R}_{\ge 0}^{n}\) with \(\sum_{i=1}^n x_i = 1\) and an \(n \times n\) symmetric matrix \(\bA = (a_{ij})_{i,j \in [n]}\) such that \(0\le a_{ij} \le 1\) and \(a_{ii} = 0\) for all \(i, j \in [n]\).
If \(e=\{i,j\}\in \binom{[n]}{2}\), we write 
\(a_e\) instead of \(a_{ij}\). 

For a weighted graph \(\GG=\GG(\vec{\mathbf{x}}, \bA)\) of order \(n\) and an integer \(r\ge1\), we define the {\it \(r\)-clique density} of \(\GG\) to be
\begin{equation}\label{eq:weighted-density}
        \GG(K_r)
        \coloneqq
        \sum_{S\in\binom{[n]}{r}}
        \prod_{e\in \binom{S}{2}}a_{e}
        \prod_{i\in S}x_i.
\end{equation}
Every \(n\)-vertex graph \(G\) can be viewed as a weighted graph \(\GG=\GG(\vec{\mathbf{x}},\bA)\) of order \(n\), where \(x_i=1/n\) for each \(i\in[n]\) and \(\bA\) is the adjacency matrix of \(G\).
Under this correspondence, \(\GG(K_r)=p_r(G)\) for all \(r\ge 2\). 

Now we state the weighted version of \Cref{THM:ordinary-main} as follows.

\begin{theorem}\label{THM:Main}
Let \(t >s \ge 2 \) be integers. Then for every weighted graph \(\GG\),
\begin{equation}\label{eq:main}
        \GG(K_t)\ge F_t\bigl(F_s^{-1}(\GG(K_s))\bigr).
\end{equation}
Moreover, this bound is sharp.
\end{theorem}

Equivalently, Theorem~\ref{THM:Main} can be written in the following explicit form:
for any integers \(t>s\ge2\) and every weighted graph \(\GG\) with \(\GG(K_s)>0\), if the integer \(p\ge s-1\) and \(\alpha\in[0,1/p)\) are (uniquely) determined by
\[
\GG(K_s)=\frac{\binom{p+1}{s}}{(p+1)^s}(1+\alpha)^{s-1}\bigl(1-(s-1)\alpha\bigr)
\implies 
\GG(K_t)\ge\frac{\binom{p+1}{t}}{(p+1)^t}(1+\alpha)^{t-1}\bigl(1-(t-1)\alpha\bigr).
\]

Our focus, however, is the following special case of Theorem~\ref{THM:Main}, which we refer to as the \emph{clique lifting}.
For each integer \(r\ge2\), define
\(\Phi_r \coloneqq F_{r+1}\circ F_r^{-1}\)
as the map from \([0,1/r!)\) to \([0,1/(r+1)!)\), where \(F_r^{-1}\) denotes the generalized inverse introduced earlier. Since both \(F_r^{-1}\) and \(F_{r+1}\) are continuous and non-decreasing, so is \(\Phi_r\).

\begin{theorem}[The clique lifting]\label{THM:clique-lifting}
Let \(r\ge2\) be an integer. Then for every weighted graph \(\GG\)
\begin{equation}\label{equ:clique-lifting}
\GG(K_{r+1})\ge \Phi_r\bigl(\GG(K_r)\bigr).
\end{equation}
\end{theorem}

We remark that Theorem~\ref{THM:clique-lifting} \(\implies\) Theorem~\ref{THM:Main} \(\implies\) Theorem~\ref{THM:ordinary-main}.\footnote{In fact, all three theorems are equivalent. The implication
Theorem~\ref{THM:Main} \(\implies\) Theorem~\ref{THM:clique-lifting}
is immediate, while
Theorem~\ref{THM:ordinary-main} \(\implies\) Theorem~\ref{THM:Main}
follows from a standard limiting argument; see~\cite[Section~2]{reiher2016clique} for details.}
To prove the first implication, note that \eqref{equ:clique-lifting} and the monotonicity of \(F_{r+1}^{-1}\) imply that 
\(F_{r+1}^{-1}\bigl(\GG(K_{r+1})\bigr)\ge F_r^{-1}\bigl(\GG(K_r)\bigr)\) for every \(r\ge 2\). Consequently, for all integers \(t>s\ge 2\),
\(F_t^{-1}\bigl(\GG(K_t)\bigr)\ge F_{t-1}^{-1}\bigl(\GG(K_{t-1})\bigr)\ge\cdots\ge F_s^{-1}\bigl(\GG(K_s)\bigr)\),
which implies \eqref{eq:main}. 
The second implication follows immediately from the correspondence between a graph and its associated weighted graph. 

\medskip

Before proceeding, we briefly review the known results for weighted graphs. The notion of weighted graphs was introduced by Nikiforov~\cite{nikiforov2011number}, who reformulated the clique density problem in this setting and proposed an analytic approach to its study. He proved the weighted clique density theorem for triangles and \(K_4\), and the general case was later established by Reiher~\cite{reiher2016clique}, thereby resolving the Lov\'asz--Simonovits conjecture in full.

\medskip

In the remainder of the paper, we prove Theorem~\ref{THM:clique-lifting} by induction on \(r\ge 2\). The base case \(r=2\) is precisely Nikiforov's theorem for triangles, which we state below. Observe that \(\Phi_2=F_3\).

\begin{theorem}[Nikiforov~\cite{nikiforov2011number}]\label{THM:Nikiforov-weighted}
For every weighted graph \(\GG\), it holds that 
\(\GG(K_3)\ge \Phi_2\bigl(\GG(K_2)\bigr).\)
\end{theorem}

\section{Proof Overview}\label{SEC:Overview}
\noindent
In this section, we outline the proof of Theorem~\ref{THM:clique-lifting}.

We first introduce the necessary notation, most of which follows that of Reiher~\cite{reiher2016clique}. Let \(\GG=\GG(\vec{\mathbf{x}},\bA)\) be a weighted graph of order \(n\). For \(I\subseteq[n]\) and \(r\ge1\), define {\it the rooted \(r\)-clique density} of \(\GG\) to be 
\begin{equation}\label{eq:local-density-def}
        \GG_I(K_r)
        \coloneqq
        \sum_{S\in \binom{[n]\setminus I}{r}}
        \left(\prod_{u\in I,\,v\in S} a_{uv}\right)
        \left(\prod_{e\in \binom{S}{2}} a_{e}\right)
        \prod_{v\in S} x_v.
\end{equation}
We write \(\GG_i(K_r)\) and \(\GG_{ij}(K_r)\) when \(I=\{i\}\) and \(I=\{i,j\}\), respectively. 
With this notation, it holds that
\begin{equation}\label{eq:derivative-identities}
        \frac{\partial \GG(K_r)}{\partial x_i}=\GG_i(K_{r-1}),
        \quad\text{and for \(r\ge 2\)}, \quad
        \frac{\partial \GG(K_r)}{\partial a_{ij}}=x_i x_j\,\GG_{ij}(K_{r-2}).
\end{equation}
Let \(\rho_i\coloneqq \GG_i(K_1)\) denote the {\it weighted degree} of the vertex \(i\in [n]\). 
If \(\rho_i>0\), we define {\it the weighted neighborhood graph}
\(\GG^{(i)}\coloneqq \GG^{(i)}(\vec{\mathbf{x}}^{(i)},\bA^{(i)})\) on the vertex set \([n]\setminus\{i\}\) such that its vertex weights and edge weights satisfy
\[
        x^{(i)}_j\coloneqq\frac{a_{ij}x_j}{\rho_i} \quad \mbox{and} \quad
        a^{(i)}_{jk}\coloneqq a_{jk}
        \quad \text{for all~} j,k\in[n]\setminus\{i\}.
\]
Since \(  \sum_{j\ne i}x^{(i)}_j
        =
        \frac{1}{\rho_i}\sum_{j\ne i}a_{ij} x_j
        =
        1,\)
this defines a weighted graph. 
Moreover, for each \(i\in [n]\) and \(r\ge 1\),
\begin{equation}\label{eq:neighborhood-graph}
    \GG^{(i)}(K_r)
        =
        \sum_{S\in\binom{[n]\setminus\{i\}}{r}}
        \prod_{e\in\binom{S}{2}}a_e
        \prod_{j\in S}\frac{a_{ij} x_j}{\rho_i}
        =
        \frac{1}{\rho_i^r}
        \sum_{S\in\binom{[n]\setminus\{i\}}{r}}
        \prod_{j\in S}a_{ij}
        \prod_{e\in\binom{S}{2}}a_e
        \prod_{j\in S}x_j
        =
        \frac{\GG_i(K_r)}{\rho_i^r}.
\end{equation}
This identity reveals an exact relationship between the \(r\)-clique density in \(\GG^{(i)}\) and the quantities \(\GG_i(K_r)\) and \(\rho_i\), which will play an important role in the subsequent proof.

\medskip

For comparison, we briefly recall the main ideas in Reiher's proof of the clique density theorem~\cite{reiher2016clique}. Since the full argument is technically involved, we only indicate the conceptual framework. 
Reiher's proof proceeds by a double induction on the clique size \(r\) and on the order \(n\) of the weighted graph. Fix \(r\), and let \(\GG\) be a counterexample of minimum order \(n\); subject to this choice, assume that
\(\GG\) maximizes the defect \(F_r(\GG(K_2))-\GG(K_r)\). 
Writing \(\gamma=\GG(K_2)\) and \(\lambda=F_r'(\gamma)\), the Lagrange multiplier condition for variations of the vertex weights gives a constant \(\mu\) such that
\(\mu=\lambda \GG_i(K_1)-\GG_i(K_{r-1})\) for every \(i\in[n]\). Starting from this relation, Reiher used a Razborov-type double-counting argument (i.e., \cite{razborov2008minimal}) to obtain the key estimate
\[
        (r-1)\GG(K_r)+(r+1)\GG(K_{r+1})
        \le \lambda\bigl(\GG(K_2)+3\GG(K_3)\bigr) -2\GG(K_2)\mu.
\]
Next, for each \(i\in[n]\), consider the weighted neighborhood graph \(\GG^{(i)}\), which has order \(n-1\). Applying the double induction to
\(\GG^{(i)}\) gives lower bounds on both \(\GG^{(i)}(K_{r-1})\) and
\(\GG^{(i)}(K_r)\). Together with the identity
\eqref{eq:neighborhood-graph} and a delicate analytic argument, these local
bounds yield a second estimate controlling
\(3\lambda\GG(K_3)-(r+1)\GG(K_{r+1})\).
Finally, adding the two estimates produces substantial cancellations and
leads to a contradiction.

\medskip

Whereas Reiher's theorem is an edge-to-clique result, Theorem~\ref{THM:clique-lifting} establishes the clique lifting from the \(K_r\)-density to the \(K_{r+1}\)-density. Its proof also has a different approach from Reiher's proof~\cite{reiher2016clique}: it uses a single induction on \(r\), rather than a double induction on the clique size and the order of the weighted graph.
The base case \(r=2\) is provided by Theorem~\ref{THM:Nikiforov-weighted}. Now fix \(r\ge3\) and let \(\GG\) be a counterexample with minimum order \(n\) and, subject to this, maximizing the defect \(\Phi_r(\GG(K_r))-\GG(K_{r+1})\). Let \(\GG(K_r)=F_r(\gamma)\). 
By the method of Lagrange multipliers, there exists a constant \(\mu\) such that \(\mu=\lambda\GG_i(K_{r-1})-\GG_i(K_r)\) for all \(i\in[n]\), where \(\lambda=\Phi_r'(x_0)\) with \(x_0=F_r(\gamma)\). Multiplying~\eqref{eq:vertex-general} by \(x_i\) and summing over \(i\in[n]\), we obtain 
\[
\mu=r\lambda\GG(K_r)-(r+1)\GG(K_{r+1}).
\]
Let \(\mu_0\) denote the corresponding extremal value of \(\mu\), namely \(\mu_0\coloneqq r\lambda F_r(\gamma)-(r+1)F_{r+1}(\gamma)\). By assumption, \(\GG(K_{r+1})<\Phi_r(\GG(K_r))=\Phi_r(F_r(\gamma))=F_{r+1}(\gamma)\)
and thus \(M\coloneqq \mu/\mu_0>1\). 
Now fix \(i\in [n]\) and consider the weighted neighborhood graph \(\GG^{(i)}\). Using \eqref{eq:neighborhood-graph} and applying the induction hypothesis to \(\GG^{(i)}\) (with \(r-1\)) yields
\begin{equation}\label{eq:outline}
\frac{\GG_i(K_r)}{\rho_i^r}=\GG^{(i)}(K_r)\geq \Phi_{r-1}(\GG^{(i)}(K_{r-1}))=\Phi_{r-1}\left(\frac{\GG_i(K_{r-1})}{\rho_i^{r-1}}\right).
\end{equation}
On the other hand, via an analytic extension technique (Lemma~\ref{LEM:envelope}), the condition \(M>1\) implies that \(\GG_i(K_r)\) is too small relative to \(\GG_i(K_{r-1})\). Combined with \eqref{eq:outline}, this forces \(\rho_i\) to be relatively large; see Claim~\ref{CLAIM:local-general} for details. Summing over all \(i\in[n]\) then gives a lower bound on \(\GG(K_2)\), while repeated clique lifting in the reverse direction yields the upper bound \(\GG(K_2)\le F_r^{-1}(\GG(K_r))\). A straightforward calculation shows that these two bounds are incompatible, thereby completing the proof.

\section{Auxiliary Lemmas}\label{SEC:Analytic}
\noindent In this section we collect several auxiliary lemmas that will be used in the proof of~\Cref{THM:clique-lifting}. 
Recall \(\Phi_r=F_{r+1}\circ F_r^{-1}\). 
Note that \(\Phi_r(x)>0\) if and only if \(x>(1/r)^r\).
By \eqref{eq:theta-intro}, we have
\begin{equation}\label{eq:theta}
    \theta_{p,r}=\frac{\binom{p+1}{r}}{(p+1)^r}=F_r\left(\frac{p}{2(p+1)}\right) \quad \mbox{for any integers \(r\ge 2\) and \(p\ge r-1\)}.
\end{equation} 

\begin{lemma}\label{LEM:ratio}
For every \(r\ge3\), the function
\begin{equation}\label{eq:ratio-function}
        u\longmapsto \frac{\Phi_{r-1}(u)}{u^{r/(r-1)}}
\end{equation}
is identically zero on \(\big(0,1/(r-1)^{r-1}\big]\)
and is strictly increasing on \(\big(1/(r-1)^{r-1},1/(r-1)!\big)\).
\end{lemma}

\begin{proof}
First consider when \( 0<u\le 1/(r-1)^{r-1}\).
Then \(F_{r-1}^{-1}(u)\le (r-2)/(2(r-1))\), and consequently
\(\Phi_{r-1}(u)=F_r(F_{r-1}^{-1}(u))=0\).  Thus the function in
\eqref{eq:ratio-function} is indeed identically zero on this interval.
 
For the strictly increasing part, we decompose the interval \(\big(1/(r-1)^{r-1},1/(r-1)!\big)\) as the union of \((\theta_{p-1,r-1},\theta_{p,r-1}]\) for \(p\ge r-1\) and verify the monotonicity on each subinterval.
Assume \(u\in (\theta_{p-1,r-1}, \theta_{p,r-1}]\) for some integer \(p\ge r-1\) and let \(\gamma\coloneqq F_{r-1}^{-1}(u)\).  Then \(\gamma\in(\theta_{p-1,2}, \theta_{p,2}]\) and 
\(\gamma=\frac{p}{2(p+1)}(1-\alpha^2)\coloneqq \gamma(\alpha)\) for some unique \(\alpha\in [0,1/p\)). 
Then 
\[
        F_{r-1}(\gamma)=
        \frac{\binom{p+1}{r-1}}{(p+1)^{r-1}}
        (1+\alpha)^{r-2}\bigl(1-(r-2)\alpha\bigr)
        \quad \mbox{and} \quad 
        F_{r}(\gamma)=
        \frac{\binom{p+1}{r}}{(p+1)^{r}}
        (1+\alpha)^{r-1}\bigl(1-(r-1)\alpha\bigr). 
\]
Since \(F_r(\gamma)\) is differentiable for \(\gamma(\alpha)\in (\theta_{p-1,2},\theta_{p,2})\), a direct differentiation (see Appendix~\ref{APP:derivatives} for details) yields 
\begin{equation}\label{eq:log-derivative}
    \frac{\dd}{\dd\alpha}
\log\left(\frac{F_r(\gamma)}{F_{r-1}(\gamma)^{r/(r-1)}}\right)
= -\frac{r\alpha}{(1+\alpha)(1-(r-2)\alpha)(1-(r-1)\alpha)}<0.
\end{equation}
Since \(F_{r-1}(\gamma)\) is
strictly decreasing as a function of \(\alpha\) on \([0,1/p)\),~\eqref{eq:log-derivative} implies that the
function in~\eqref{eq:ratio-function} is strictly increasing as a function of
\(u=F_{r-1}(\gamma)\) on the interval \((\theta_{p-1,r-1}, \theta_{p,r-1})\).  Since \(\Phi_{r-1}\) is continuous, the strict monotonicity holds on the whole interval
\(\big(1/(r-1)^{r-1},1/(r-1)!\big)\).  This completes the proof.
\end{proof}

Note that the function \eqref{eq:F-def} (on its strictly increasing range) can be viewed as a multi-variable function with variables $r-1\leq p\in \mathbb{N}$ and $0\leq \alpha \leq 1/p$.
As a key step in the proof of~\Cref{THM:clique-lifting}  (i.e., in Claim~\ref{CLAIM:local-general}), we need to extend this function to the interval $0\leq \alpha\leq 1/(r-1)$. 
For this purpose, we define the following extension: for fixed integers \(q \ge r\ge 2\),
\begin{equation}\label{equ:P_r}
        P_r(q,\beta)\coloneqq
        \frac{\binom q r}{q^r}(1+\beta)^{r-1}
        \bigl(1-(r-1)\beta\bigr) \quad \mbox{where} \quad 0\le\beta\le 1/(r-1).
\end{equation} 
Let \(q\ge r+1\). It is clear that the parametric curve \((P_r(q,\beta), P_{r+1}(q,\beta))\) coincides with \(\Phi_r\) for \(0\le \beta \le 1/(q-1)\). The next lemma shows that it actually remains below \(\Phi_r\) for all \(0\le \beta \le 1/r\).

\begin{lemma}\label{LEM:envelope}
Let \(r\ge 2\) and \(q\ge r+1\).  For every \(\beta \in[0,1/r]\),
\begin{equation}\label{eq:envelope}
        \Phi_{r}(P_r(q,\beta))
        \ge P_{r+1}(q,\beta).
\end{equation}
\end{lemma}

\begin{proof}
Fix \(r\ge 2\).
We use two elementary facts.  First, for \(\widetilde{q}\in\{q-1,q\}\) and \(\beta \in[0,1/r]\),
\begin{equation*}\label{eq:Pr-log-derivative}
        \frac{\partial}{\partial \beta}\log P_r(\widetilde{q},\beta)
        =-\frac{r(r-1)\beta}{(1+\beta)(1-(r-1)\beta)}\le0, 
\end{equation*}
with equality if and only if \(\beta=0\).
Hence, \(P_r(\widetilde{q},\cdot)\) is strictly decreasing for \(\beta\in [0,1/r]\). Second, for fixed \(r\)
and \(\beta\), since the factor \(\binom{n}{r}/n^r=\frac1{r!}\prod_{\ell=0}^{r-1}\left(1-\frac{\ell}{n}\right)\) is strictly increasing in \(n\), we have \(P_r(q-1,\beta)< P_r(q,\beta)\).

Now we prove the result by induction on \(q\).  If \(q=r+1\), then \(1/r=1/(q-1)\), so~\eqref{eq:envelope} holds with equality by the definition of \(F_r\) and \(F_{r+1}\).
Assume now \(q \geq r+2\) and that the result holds for \(q-1\).  If
\(0\le\beta \le1/(q-1)\), then~\eqref{eq:envelope} is again an equality. Hence we may assume
\( 1/(q-1)< \beta \le 1/r \). Let \(x_{0} \coloneqq P_r(q-1,0)=P_r\left(q,1/(q-1)\right)\).  

We claim that there exists a unique \(\beta'\in (0,1/r)\) such that
\begin{equation}\label{eq:beta-prime}
        \beta'<\beta \quad \text{and} \quad x_1 \coloneqq P_r(q-1,\beta')=P_r(q,\beta).
\end{equation}
To see it, using the above two elementary facts, we have \(x_{0}=P_r(q,1/(q-1))>P_r(q,\beta)\), and further
\[
    P_r(q-1,1/r)<P_r(q,1/r)\le P_r(q,\beta)<x_{0}=P_r(q-1,0).
\]
By the intermediate value theorem, there exists \(\beta'\in (0,1/r)\) such that \(P_r(q-1,\beta')=P_r(q,\beta)\).  
Moreover, we have \(P_r(q-1,\beta)<P_r(q,\beta)=P_r(q-1,\beta')\);
since \(P_r(q-1,\cdot)\) is strictly decreasing, we derive \(\beta'<\beta\), as claimed.

For \(\widetilde{q}\in \{q-1,q\}\), let \(Z_{\widetilde{q}}(x)\) denote \(P_{r+1}(\widetilde{q},\beta)\) written as a function of
\(x=P_r(\widetilde{q},\beta)\). Then we have
\begin{equation}\label{eq:slope}
        Z_{\widetilde{q}}'(x)=\frac{\dd P_{r+1}(\widetilde{q},\beta)/\dd \beta}{\dd P_r(\widetilde{q},\beta )/\dd \beta}
        =\frac{\widetilde{q}-r}{\widetilde{q}(r-1)}(1+\beta).\footnote{We refer to Appendix~\ref{APP:derivatives} for a detailed verification.}
\end{equation}
Recall the definition of \(x_0\).
We first note that \(Z_q(x_0)=P_{r+1}(q, 1/(q-1))=P_{r+1}(q-1,0)=Z_{q-1}(x_0)\). 
For \(x<x_0\), define parameters \(\beta_0\) and \(\beta'_0\) so that \(P_r(q,\beta_0)=P_r(q-1,\beta'_0)=x\).
Since \(P_r(q-1,\beta_0)<P_r(q,\beta_0)=P_r(q-1,\beta'_0)\), we have \(\beta_0>\beta'_0\). 
Then using~\eqref{eq:slope}, we see \(Z_q'(x)>Z_{q-1}'(x)\ge 0\) for $x<x_0$. By the Newton--Leibniz formula and the fact \(Z_q(x_0)=Z_{q-1}(x_0)\), we have \(Z_q(x)\le Z_{q-1}(x)\) for every \(x<x_{0}\). In particular, since \(x_{1}=P_r(q,\beta)\le x_{0}\),
\begin{equation}\label{eq:beta-beta'}
    P_{r+1}(q,\beta)=Z_q(x_{1})\le Z_{q-1}(x_{1})=P_{r+1}(q-1,\beta').
\end{equation}
Note that \(q-1\geq r+1\). Using induction on \(q-1\), we complete the proof by deriving the desired inequality
\[
        \Phi_{r}(P_r(q,\beta))
        \overset{\eqref{eq:beta-prime}}{=}\Phi_{r}(P_r(q-1,\beta'))
        \ge P_{r+1}(q-1,\beta')
        \overset{\eqref{eq:beta-beta'}}{\ge} P_{r+1}(q,\beta). \qedhere
\]
\end{proof}

The next lemma is a weighted graph version of a classic clique density inequality (see~\cite{lovasz1983number}). 
For completeness, we include a self-contained proof in Appendix~\ref{APP:quadratic}. 

\begin{lemma}\label{LEM:quadratic-MM}
For every integer \(r\ge2\) and every weighted graph \(\GG\),
\[
        r^2\GG(K_r)^2
        \le
        \GG(K_{r-1})\Bigl(\GG(K_r)+(r^2-1)\GG(K_{r+1})\Bigr).
\]
\end{lemma}

In the following lemma, we prove Theorem~\ref{THM:clique-lifting} for all critical points (i.e., when \(\GG(K_r)=\theta_{p,r}\) for \(p\ge r\ge 3\)), assuming that Theorem~\ref{THM:clique-lifting} holds for \(r-1\). 

\begin{lemma}\label{LEM:critical-values}
Let \(p\ge r \ge 3\) and assume Theorem~\ref{THM:clique-lifting} holds for \(r-1\).   If \(\GG\) is a weighted graph with \(\GG(K_r)=\theta_{p,r}\), then \(\GG(K_{r+1})\ge \theta_{p,r+1}\).
\end{lemma}

\begin{proof}
By the assumption that Theorem~\ref{THM:clique-lifting} holds for \(r-1\), we have \(\theta_{p,r}=\GG(K_r)\ge \Phi_{r-1}(\GG(K_{r-1})).\)
Since \(\Phi_{r-1}\) is non-decreasing, we obtain 
\(\GG(K_{r-1})\le \Phi_{r-1}^{-1}\left(\theta_{p,r}\right)=\theta_{p,r-1}.
\)
By Lemma~\ref{LEM:quadratic-MM}, we have
\[
        r^2\GG(K_r)^2
        \le
        \GG(K_{r-1})\Bigl(\GG(K_r)+(r^2-1)\GG(K_{r+1})\Bigr).
\]
As \(\GG(K_r)=\theta_{p,r}>0\) and \(\GG(K_{r-1})\le \theta_{p,r-1}\), it follows that
\(
r^2\theta_{p,r}^2\le\theta_{p,r-1}\bigl(\theta_{p,r}+(r^2-1)\GG(K_{r+1})\bigr).
\)
Therefore,
\[
        \GG(K_{r+1})
        \ge
        \frac{r^2\theta_{p,r}^2/\theta_{p,r-1}-\theta_{p,r}}{r^2-1}
        =
        \frac{\binom{p+1}{r}}{(p+1)^r}\cdot
        \frac{p+1-r}{(p+1)(r+1)}
        =
        \frac{\binom{p+1}{r+1}}{(p+1)^{r+1}}
        =\theta_{p,r+1}. \qedhere
\]
\end{proof}

In the final lemma of this section, we gather several useful properties of a minimal counterexample to Theorem~\ref{THM:clique-lifting}, which will be used later. 

\begin{lemma}\label{LEM:reduction}
Let \(r\ge3\) and assume that Theorem~\ref{THM:clique-lifting} holds for \(r-1\).  Suppose that there is a weighted graph \(\GG_0\) such that \(\Phi_r(\GG_0(K_r))>0\) and \(\GG_0(K_{r+1})<\Phi_r(\GG_0(K_r))\).
Then there exists a weighted graph \(\GG=\GG(\vec{\mathbf{x}}, \bA)\) satisfying the two inequalities above and the following properties:
\begin{enumerate}[label={\(\arabic*\)}]
    \item[(1).] All \(x_i\)'s are positive;
    \item[(2).] \(\GG(K_r)=F_r(\gamma)\) for some $\gamma=\frac{p}{2(p+1)}(1-\alpha^2)$ with $r\le p\in\mathbb{N}$ and $0<\alpha<1/p$;
    \item[(3).] \(\GG\) is a local maximizer of \(\Phi_r(\GG(K_r))-\GG(K_{r+1})\), subject to the constraints \(\sum_i x_i=1\) and \(0\le a_{ij}\le 1\).
\end{enumerate} 
\end{lemma}

\begin{proof}
For any weighted graph \(\GG\), let
\[
    f(\GG)\coloneqq\Phi_r(\GG(K_r))-\GG(K_{r+1}).
\]
Among all weighted graphs \(\GG\) satisfying \(\Phi_r(\GG(K_r))>0\) and \(f(\GG)>0\), choose one of minimum order \(n\).  By compactness, the continuous function \(f\) attains its maximum over the space of weighted graphs of order \(n\). Let \(\GG\) be such a maximizer and clearly \(f(\GG)>0\).

First, if \(x_i=0\) for some \(i\in [n]\), deleting the vertex \(i\) leaves all clique densities, and hence the value of \(f\), unchanged. Thus the resulting weighted graph \(\GG'\) satisfies \(\Phi_r(\GG'(K_r))>0\) and \(f(\GG')>0\), contradicting the minimal choice of \(n\).  So every vertex weight of \(\GG\) is positive.

If \(\Phi_r(\GG(K_r))=0\), then \(f(\GG)=-\GG(K_{r+1})\le0 \),
contrary to \(f(\GG)>0\).
Hence we have \(\Phi_r(\GG(K_r))>0\).
If \(\GG(K_r)=\theta_{p,r}\), then Lemma~\ref{LEM:critical-values} gives
\[
        \GG(K_{r+1})\ge \theta_{p,r+1}
        =
        \Phi_r(\theta_{p,r})
        =
        \Phi_r(\GG(K_r)),
\]
and hence \(f(\GG)\le 0\), again a contradiction. So \(\GG(K_r)\) cannot equal \(\theta_{p,r}\) for any \(p\ge r\). 
So \( \gamma\coloneqq F_r^{-1}(\GG(K_r))\in
        \left(\frac{r-1}{2r},\frac12\right)
        \setminus
        \left\{\frac{p}{2(p+1)}:p\ge r\right\}.\)
Equivalently, \(\GG(K_r)=F_r(\gamma)\) with \(\gamma=\frac{p}{2(p+1)}(1-\alpha^2)\), where \(r\leq p\in \mathbb{N}\) and \(0<\alpha<1/p\) are unique.
Therefore, \(\Phi_r\) is differentiable at \(\GG(K_r)\), and the global maximizer is in particular a local maximum of \(f\) subject to the constraint \(\sum_i x_i=1\) and \(0\le a_{ij}\le 1\). This completes the proof.
\end{proof}

\section{Proof of Theorem~\ref{THM:clique-lifting}}\label{SEC:MainProof}
\noindent In this section, we aim to prove Theorem~\ref{THM:clique-lifting} by induction on \(r\). That is, \(\GG(K_{r+1})\ge \Phi_r(\GG(K_r))\) holds for every \(r\ge2\) and every weighted graph \(\GG\).
The base case \(r=2\) is precisely~\Cref{THM:Nikiforov-weighted}. 
Throughout the remainder of this section, we assume \(r\ge3\) and that the statement holds for all smaller values. 

Suppose, for a contradiction, that \eqref{equ:clique-lifting} fails for some weighted graph. 
By Lemma~\ref{LEM:reduction}, there exists a weighted graph \(\GG=\GG(\vec{\mathbf{x}}, \bA)\) of order \(n\), 
which attains a local maximum of \(\Phi_{r}(\GG(K_r))-\GG(K_{r+1})>0\), has all vertex weights positive, and admits parameters \(p\ge r\) and \(0<\alpha<1/p\) such that, writing \(\gamma=\frac{p}{2(p+1)}(1-\alpha^2)\), 
we have \(\GG(K_r)=F_r(\gamma)\) and, hence, \(\GG(K_{r+1})< \Phi_r(\GG(K_r))=\Phi_r(F_r(\gamma))=F_{r+1}(\gamma)\). 
At the point \(x_0\coloneqq F_r(\gamma)\), differentiating \(\Phi_r=F_{r+1}\circ F_r^{-1}\) via the chain rule and the \(\alpha\)-parametrization (see Appendix~\ref{APP:derivatives} for details) yields
\begin{equation}
        \lambda=\Phi_r'(x_0)=\frac{p-r+1}{(r-1)(p+1)}(1+\alpha). \label{eq:value-of-lambda}
\end{equation}

Since \(\GG\) is a local maximum of \(\Phi_r(\GG(K_r))-\GG(K_{r+1})\) under \(\sum_i x_i=1\), Lagrange multipliers, together with the first equation of \eqref{eq:derivative-identities}, yield that there exists a constant \(\mu\) such that 
\begin{equation}\label{eq:vertex-general}
        \GG_i(K_{r})=\lambda \GG_i(K_{r-1})-\mu,
\end{equation}
for all \(i\in [n]\). 
By multiplying~\eqref{eq:vertex-general} by \(x_i\) and summing over \(i\in [n]\), we derive that
\[
        \mu=r\lambda \GG(K_r)-(r+1)\GG(K_{r+1}).
\]
\noindent\text{Now we set}\hfill
\(\displaystyle
\mu_0 \coloneqq  r\lambda F_r(\gamma)-(r+1)F_{r+1}(\gamma)
= \frac{r+1}{r-1}
\frac{1}{(p+1)^{r+1}}\binom{p+1}{r+1}(1+\alpha)^r.
\)\hfill\null

\smallskip
\noindent Since \(\GG(K_r)=F_{r}(\gamma)\) and \(\GG(K_{r+1})<F_{r+1}(\gamma)\), we obtain \(\mu>\mu_{0}>0\). Define
\begin{equation*}
        M\coloneqq\frac{\mu}{\mu_{0}}>1 \quad \text{and}\quad \eta_i\coloneqq \frac{\lambda}{\mu}\cdot \GG_i(K_{r-1}).
\end{equation*}

We need a crucial estimate for the weighted degree $\rho_i\coloneqq\GG_i(K_1)$ for every vertex $i\in [n]$ (see Claim~\ref{CLAIM:local-general}) for later proofs.
To proceed, we first provide the following estimates on the value of \(\eta_{i}\). 

\begin{claim}\label{CLAIM:range-of-eta}
    For every vertex \(i\in [n]\), we have \(1\le\eta_i\le r\).
\end{claim}
\begin{proof}
    Since \(\mu(\eta_{i}-1)=\GG_i(K_r)\ge 0\) and \(\mu>0\), it follows that \(\eta_i\ge1\). 
    If \(a_{ij}>0\), since \(\GG\) is a local maximum, decreasing \(a_{ij}\) cannot increase \(\Phi_r(\GG(K_{r}))-\GG(K_{r+1})\). 
    Using the second equation of \eqref{eq:derivative-identities}, we obtain 
    \begin{equation}\label{eq:edge-general}
        \GG_{ij}(K_{r-1})\le \lambda \GG_{ij}(K_{r-2}) \quad \text{for all} \quad  a_{ij}>0.
    \end{equation}
    By multiplying~\eqref{eq:edge-general} by \(x_{j}a_{ij}\) and summing over all \(j\), we derive that
\[
        r\mu (\eta_{i}-1)=r \GG_i(K_r)
        =\sum_{j\neq i} x_j a_{ij}\GG_{ij}(K_{r-1})
        \le \lambda\sum_{j\neq i} x_j a_{ij}\GG_{ij}(K_{r-2})
        =(r-1)\lambda \GG_i(K_{r-1})=(r-1)\mu \eta_{i}, 
\]
which implies \(\eta_i\le r\). This completes the proof of Claim~\ref{CLAIM:range-of-eta}. 
\end{proof}

Now we are ready to prove the following lower bound on the weighted degree \(\rho_{i}=\GG_i(K_1)\). 

\begin{claim}\label{CLAIM:local-general}
For every vertex \(i\in [n]\), we have
\begin{equation}
        \rho_{i}\ge
        \frac{p(1+\alpha)}{p+1}M^{1/(r-1)}
        \cdot \frac{\eta_i+r(r-2)}{r(r-1)}. \label{eq:lower-bound-of-degree}
\end{equation}
\end{claim}

\begin{proof}
Fix \(i\in [n]\). It suffices to show that \(\rho_{i}\ge \tau_{i}\), where we set
\begin{equation*}
        \tau_{i}\coloneqq\frac{p(1+\alpha)}{p+1}M^{1/(r-1)}\cdot \frac{\eta_i+r(r-2)}{r(r-1)}.
\end{equation*}
 
If \(\rho_i=0\), then \(\GG_i(K_{r-1})=\GG_i(K_r)=0\), and the equation \eqref{eq:vertex-general} forces \(\mu=0\), contradicting \(\mu>\mu_0>0\). 
Hence, we may assume \(\rho_i>0\).
Consider the weighted neighborhood graph \(\GG^{(i)}\). 
Applying the induction hypothesis to \(\GG^{(i)}\) together with~\eqref{eq:neighborhood-graph}, we obtain 
\begin{equation}
        \frac{\GG_i(K_r)}{\rho_{i}^r}= \GG^{(i)}(K_r)\ge \Phi_{r-1}(\GG^{(i)}(K_{r-1})) = 
        \Phi_{r-1}\left(\frac{\GG_i(K_{r-1})}{\rho_{i}^{r-1}}\right). \label{eq:neighborhood-inq}
\end{equation}
Now define
\begin{equation*}
        \beta\coloneqq\frac{r-\eta_{i}}{\eta_{i}+r(r-2)}.
\end{equation*}
By Claim~\ref{CLAIM:range-of-eta},
we have \(1\le \eta_{i}\le r\), and hence \(0\le\beta\le1/(r-1)\). 
The parameter \(\beta\) is introduced so that the quantities \(\tau_i, \GG_i(K_{r-1}), \GG_i(K_r)\) admit closed algebraic expressions in terms of the functions $P_{r-1}(p,\beta)$ and $P_r(p,\beta)$ defined in \eqref{equ:P_r}. 
More precisely, we have (see Appendix~\ref{APP:derivatives} for the derivation) 
\begin{equation}\label{eq:key-eq-clique-lifting}
        \frac{\GG_i(K_{r-1})}{\tau_{i}^{r-1}}
        =P_{r-1}(p,\beta)
        \quad \text{and} \quad 
        \frac{\GG_i(K_r)}{\tau_{i}^r}
        =M^{-1/(r-1)}P_r(p,\beta). 
\end{equation}

Now we split the remaining proof of this claim into two cases based on the value of \(\eta_{i}\). Note that $\eta_i\geq 1$.

\medskip

\noindent\underline{\textbf{The case when \(\eta_{i}>1\):}} In this case, we have \(\GG_i(K_r)=\mu(\eta_{i}-1)>0\). 
By \eqref{eq:key-eq-clique-lifting} and Lemma~\ref{LEM:envelope}, it follows that  
\begin{equation*}
        \Phi_{r-1}\left(\frac{\GG_i(K_{r-1})}{\tau_{i}^{r-1}}\right)
        =\Phi_{r-1}(P_{r-1}(p,\beta))
        \ge P_r(p,\beta)= M^{1/(r-1)}\frac{\GG_i(K_r)}{\tau_{i}^r}> \frac{\GG_i(K_r)}{\tau_{i}^r},
\end{equation*}
where the last strict inequality holds because $M>1$ and $\GG_i(K_r)>0$.
If \(\rho_{i}<\tau_{i}\), then \(\GG_i(K_{r-1})/\rho_{i}^{r-1}>\GG_i(K_{r-1})/\tau_{i}^{r-1}\). 
By Lemma~\ref{LEM:ratio}, the mapping \( u\longmapsto \frac{\Phi_{r-1}(u)}{u^{r/(r-1)}}\) is non-decreasing and thus we have
\[
        \rho_{i}^r\Phi_{r-1}\left(\frac{\GG_i(K_{r-1})}{\rho_{i}^{r-1}}\right)
        =\GG_i(K_{r-1})^{r/(r-1)}
        \frac{\Phi_{r-1}(\GG_i(K_{r-1})/\rho_{i}^{r-1})}{(\GG_i(K_{r-1})/\rho_{i}^{r-1})^{r/(r-1)}}
        \geq 
        \tau_{i}^r\Phi_{r-1}\left(\frac{\GG_i(K_{r-1})}{\tau_{i}^{r-1}}\right)> \GG_i(K_r).
\]
This contradicts~\eqref{eq:neighborhood-inq}. Therefore, indeed we have \(\rho_{i}\ge \tau_{i}\) when \(\eta_{i}>1\). 

\medskip

\noindent\underline{\textbf{The case when \(\eta_{i}=1\):}}  In this case, we have \(\GG_i(K_r)=0\) and \(\beta=1/(r-1)\).  If \(\rho_{i}<\tau_{i}\), then
\[
        \frac{\GG_i(K_{r-1})}{\rho_{i}^{r-1}}>\frac{\GG_i(K_{r-1})}{\tau_{i}^{r-1}}
        =P_{r-1}\left(p,\frac1{r-1}\right) \ge \frac{1}{(r-1)^{r-1}}, 
\]
where the equation holds by \eqref{eq:key-eq-clique-lifting}
and the last inequality holds since \(P_{r-1}\left(p,1/(r-1)\right)\) is non-decreasing in \(p\in [r,\infty)\) and it is equal to \(1/(r-1)^{r-1}\) at \(p=r\). 
Since $\Phi_{r-1}(\cdot)$ is strictly increasing on $\big[1/(r-1)^{r-1},1/(r-1)!\big)$, 
\[
    \Phi_{r-1}\left(\frac{\GG_i(K_{r-1})}{\rho_{i}^{r-1}}\right)>\Phi_{r-1}\left(\frac{1}{(r-1)^{r-1}}\right)=0 = \frac{\GG_i(K_r)}{\rho_{i}^{r}},
\]
which contradicts~\eqref{eq:neighborhood-inq} again.  
Therefore \(\rho_{i}\ge \tau_{i}\) also when \(\eta_{i}=1\). 
This completes the proof of Claim~\ref{CLAIM:local-general}.
\end{proof}

\medskip

We now finish the proof. Using the definition of \(\eta_{i}, \GG_i(K_{r-1}), F_{r}(\gamma), \mu_{0}\) and the equation~\eqref{eq:value-of-lambda}, 
\begin{align*}
         \sum_i x_i\eta_i
         =&\frac{\lambda}{\mu}\sum_i x_i \GG_i(K_{r-1})
         =\frac{r\lambda \GG(K_r)}{\mu}
         =\frac{r}{M}\cdot \lambda\cdot F_r(\gamma)\Big/\mu_0\\
         =& \frac{r}{M}\cdot \frac{(p-r+1)(1+\alpha)}{(r-1)(p+1)}\cdot \frac{\binom{p+1}{r}}{(p+1)^r}
        (1+\alpha)^{r-1}\bigl(1-(r-1)\alpha\bigr)\Big/\left( \frac{r+1}{r-1}
        \frac{\binom{p+1}{r+1}}{(p+1)^{r+1}}(1+\alpha)^r \right) \\
        =&\frac{r(1-(r-1)\alpha)}{M}, 
\end{align*}
where the last equation uses the identity \(\binom{p+1}{r+1}=\frac{p-r+1}{r+1}\binom{p+1}{r}\).
Multiplying \eqref{eq:lower-bound-of-degree} by \(x_i\) and summing over all $i\in [n]$, together with the above equation, we derive 
\begin{equation}\label{eq:K2-lower}
        2\GG(K_2)=\sum_{i}x_{i}\rho_{i}\ge
        \frac{p(1+\alpha)}{(p+1)(r-1)}
        \left((r-2)M^{1/(r-1)}+
        \frac{1-(r-1)\alpha}{M^{(r-2)/(r-1)}}\right).
\end{equation}
On the other hand, using the induction hypothesis repeatedly, we obtain 
\begin{align*}
    F_{r}^{-1}(\GG(K_r))\ge F_{r-1}^{-1}(\GG(K_{r-1}))\ge F_{r-2}^{-1}(\GG(K_{r-2}))
    \ge\cdots\ge F_{2}^{-1}(\GG(K_{2}))=\GG(K_2).
\end{align*}
Since \(\GG(K_r)=F_r(\gamma)\), it follows that 
\begin{equation}\label{eq:k2-upper}
        2\GG(K_2)\le 2\gamma=\frac{p}{p+1}(1-\alpha^2).
\end{equation}
Combining~\eqref{eq:K2-lower} and~\eqref{eq:k2-upper} and canceling the positive factor \(p(1+\alpha)/(p+1)\), we obtain that
\begin{equation}\label{equ:h(M)}
    h(M)\coloneqq (r-2)M^{1/(r-1)}+
        \frac{1-(r-1)\alpha}{M^{(r-2)/(r-1)}}
        \le (r-1)(1-\alpha). 
\end{equation}
Note that $M>1$, and 
the function $h(M)$ is strictly increasing for any $M>1$, since 
\[
        \frac{\dd h(M)}{\dd M}=\frac{r-2}{r-1}M^{-(r-2)/(r-1)}
        \left(1-\frac{1-(r-1)\alpha}{M}\right)
        >0 \quad \text{whenever} \quad M>1.
\]
Hence, we have \(h(M)>h(1)=(r-1)(1-\alpha)\), a contradiction to \eqref{equ:h(M)}. 
The proof of~\Cref{THM:clique-lifting} is complete.  
\qed

\section{Concluding Remarks}\label{SEC:concluding}
\noindent We determine the asymptotically sharp lower bound on the \(t\)-clique density \(p_t(G)\) in graphs \(G\) with a given \(s\)-clique density \(p_s(G)\), for all \(2\le s<t\). 
In the case \(s=2\), this provides a new proof of the clique density theorem of Reiher~\cite{reiher2016clique}. 
We conclude by showing how the proof of our main result can be adapted to obtain a stability result: if
\(
p_t(G)\le F_t\bigl(F_s^{-1}(p_s(G))\bigr)+o(1),
\)
then \(G\) must be close to a family of well-structured graphs.
We also discuss a related general problem.

\subsection{Stability of Theorem~\ref{THM:ordinary-main}}
\noindent
Stability versions of the clique density theorem have been studied extensively; see \cite{pikhurko2017asymptotic,liu2020exact} for the triangle density case and \cite{kim2020asymptotic} for the general clique density setting.
Let \(t\ge 3\) be an integer. 
To proceed, we need to first describe the family \(\mathcal{H}_{t,n}\) of \(n\)-vertex graphs introduced in~\cite{kim2020asymptotic}.
For \(\gamma=1/2\), let
\(\mathcal H_{\gamma,n}\coloneqq\{K_n\}\).  
For \(0\le\gamma<1/2\), let \(1\le p\in \mathbb{N}\) and \(c\in(1/(p+1),1/p]\) be the unique parameters satisfying that
\[
        \gamma=\frac12\Big(1-\big(pc^2+(1-pc)^2\big)\Big).
\]
Let \(V_1,\ldots,V_{p+1}\) be a partition of an \(n\)-vertex set with \(|V_1|=\cdots=|V_p|=\lfloor cn\rfloor\) and \(|V_{p+1}|=n-p\lfloor cn\rfloor\), and put \(U\coloneqq V_p\cup V_{p+1}\).  Let \(\mathcal H_{\gamma,n}\) be the family
of all graphs obtained from the complete \(p\)-partite graph with parts
\(V_1,\ldots,V_{p-1},U\) by adding inside \(U\) an arbitrary triangle-free
graph with exactly \(|V_p||V_{p+1}|\) edges.  
Finally, let \(\mathcal H_{t,n}\) be the union of \(\mathcal H_{\gamma,n}\) over all \(0\le \gamma\le 1/2\) and the family of \(n\)-vertex \(K_t\)-free graphs. 

Now we state and prove a stability version of Theorem~\ref{THM:ordinary-main}. It characterizes graphs asymptotically minimizing \(K_t\) for a given \(s\)-clique density: every such graph is close in edit distance to a graph in \(\mathcal{H}_{t,n}\).

\begin{theorem}\label{THM:ordinary-stability}
Let \(2\le s<t\). For every \(\varepsilon>0\), there are \(\delta>0\) and \(n_0\) such that the following
holds for every graph \(G\) on \(n\ge n_0\) vertices.  If
\[
        p_t(G)\le F_{t} \big ( F_{s}^{-1}(p_s(G)) \big ) +\delta,
\]
then \(G\) can be made isomorphic to some graph in \(\mathcal{H}_{t,n}\) by adding or deleting at most \(\varepsilon n^{2}\) edges. 
\end{theorem}

In what follows, we show how the proof of Theorem~\ref{THM:ordinary-main} can be adapted to establish this result. We work in the language of graphons; for standard terminology, see the monograph of Lov\'asz~\cite{lovasz2012large}. Let \(\mathcal H_t\) denote the graphon analogue of \(\mathcal{H}_{t,n}\), viewed as a family of weak isomorphism classes of graphons.
We shall use the graphon versions of Theorem~\ref{THM:Main} (i.e., Theorem~\ref{THM:clique-lifting}) and Lemma~\ref{LEM:quadratic-MM}. These follow routinely from the standard step-graphon approximation argument, so we omit the details. In particular, every graphon \(W\) satisfies
\begin{equation}\label{eq:graphon-version} 
p_t(W)\ge F_t\bigl(F_s^{-1}(p_s(W))\bigr)\quad \text{for all} \quad t>s\ge 2.\footnote{Here and throughout, we let \(p_r(W)\coloneqq t(K_r,W)/r!\) for a graphon \(W\).}
\end{equation}
We also need the following lemma, which reduces the general clique-to-clique stability to the edge-to-clique case.

\begin{lemma}\label{lem:rigidity}
Let \(r\ge 2\) and \(W\) be a graphon. If \(p_{r+1}(W)=\Phi_r(p_r(W))>0\), then \(p_r(W)=F_r(p_2(W))\).
\end{lemma}

\begin{proof}
The case \(r=2\) is immediate. Assume \(r\ge 3\) and \(p_{r+1}(W)=\Phi_r(p_r(W))>0\). 
If \(p_r(W)=\theta_{p,r}\) for some $p\ge r$, then \(p_{r+1}(W)=\Phi_r(\theta_{p,r})=\theta_{p,r+1}\). 
Using the graphon version of Lemma~\ref{LEM:quadratic-MM}, we obtain \(p_{r-1}(W)\ge \theta_{p,r-1}\). 
Applying \eqref{eq:graphon-version} with $(t,s)=(r,r-1)$, we obtain \(p_{r-1}(W)\le \theta_{p,r-1}\).
So \(p_{r-1}(W)=\theta_{p,r-1}\). 
Iterating this argument leads to \(p_2(W)=\theta_{p,2}\). Consequently, \(p_{r}(W)=\theta_{p,r}=F_r(\theta_{p,2})=F_r(p_2(W))\), as desired. 

Now suppose that \(p_r(W)\in (\theta_{p-1,r}, \theta_{p,r})\) for some integer \(p\ge r\). 
Let \(\gamma \coloneqq F_r^{-1}(p_r(W))\). Repeating the argument in the proof of Theorem~\ref{THM:clique-lifting}, with 
\(\lambda=\Phi_r'(p_r(W))\), gives
\[
\mu=r\lambda p_r(W)-(r+1)p_{r+1}(W)
   =r\lambda F_r(\gamma)-(r+1)F_{r+1}(\gamma)=\mu_0.
\]
Thus \(M=\mu/\mu_0=1\), and the lower bound~\eqref{eq:K2-lower} yields \(p_2(W)\ge \frac{p}{2(p+1)}(1-\alpha^2)=\gamma=F_r^{-1}(p_r(W))\). So $p_r(W)\le F_r(p_2(W))$. Applying~\eqref{eq:graphon-version} with \((t,s)=(r,2)\), we have \(p_r(W)=F_r(p_2(W))\), completing the proof. 
\end{proof}

With this lemma in hand, we are ready to present the proof of Theorem~\ref{THM:ordinary-stability}.

\begin{proof}[\bf Proof of Theorem~\ref{THM:ordinary-stability}.]
Fix \(t>s\ge 2\).
Suppose this fails. Then there exist \(\varepsilon>0\) and a sequence of graphs \(G_n\) with \(|V(G_n)|\to\infty\) such that
\(p_t(G_n) \le F_t\bigl(F_s^{-1}(p_s(G_n))\bigr)+o(1),
\)
yet each \(G_n\) is at edit distance at least \(\varepsilon|V(G_n)|^2\) from every graph in \(\mathcal H_{t,|V(G_n)|}\).
By passing to a subsequence if necessary, let \(W\) be the limit graphon of the sequence of graphs \(G_{n}\). Then we have \(p_t(W) \le F_t\bigl(F_s^{-1}(p_s(W))\bigr)\).
Together with~\eqref{eq:graphon-version}, we have 
\begin{equation}\label{equ:p(W)}
    p_t(W) = F_t\bigl(F_s^{-1}(p_s(W))\bigr).
\end{equation}

Let \(\gamma \coloneqq F_s^{-1}(p_s(W))\).
If \(F_t(\gamma)=0\), then \(p_t(W)=0\), and thus \(G_n\) contains \(o(|V(G_n)|^t)\) copies of \(K_t\).  
By the Graph Removal Lemma~\cite{erdos1986asymptotic}, \(G_n\) can be made \(K_t\)-free after deleting \(o(|V(G_n)|^2)\) edges.  Since all \(K_t\)-free graphs belong to \(\mathcal H_{t,|V(G_n)|}\), this is a contradiction.

Now we may assume that \(F_t(\gamma)>0\). By~\eqref{eq:graphon-version} and the monotonicity of $F_r$, we have 
\[
F_t^{-1}(p_t(W))\ge F_{t-1}^{-1}(p_{t-1}(W))\ge\cdots\ge F_{s+1}^{-1}(p_{s+1}(W))\ge F_s^{-1}(p_s(W)).
\]
By \eqref{equ:p(W)}, we have \(F_t^{-1}(p_t(W))=F_s^{-1}(p_s(W))=\gamma>0\).
Thus every intermediate inequality above must be an equality. In particular, \(p_{s+1}(W)=\Phi_s(p_s(W))>0\).  
Using Lemma~\ref{lem:rigidity}, we have \(p_s(W)=F_s(p_2(W))\) and thus \(p_t(W)=F_t\bigl(F_s^{-1}(p_s(W))\bigr)=F_t(p_2(W)).\)
We now use the characterization of extremal graphons for the clique density theorem: \cite[Theorem~2.1]{pikhurko2017asymptotic} (for \(t=3\)) and \cite[Theorem~1.6]{kim2020asymptotic} (for \(t\ge 3\)) show that
\[
p_t(W)=F_t(p_2(W))
\iff
[W]\in\mathcal H_t
\quad \text{for~} t\ge 3,
\]
where \([W]\) denotes the weak isomorphism class of \(W\). It follows from \([W]\in \mathcal{H}_t\) that there exist graphs \(H_n\in\mathcal H_{t,|V(G_n)|}\) such that, up to relabeling,
\(|E(G_n)\triangle E(H_n)|=o\bigl(|V(G_n)|^2\bigr).\)\footnote{For this implication, we refer the reader to the proof of \cite[Theorem~1.2]{kim2020asymptotic}.}
This is a contradiction.
\end{proof}

\subsection{A General Problem}

\noindent
For a fixed graph \(F\), let \(k_F(G)\) denote the number of unlabeled copies of \(F\) in a graph \(G\), and let $p_F(G)\coloneqq\frac{k_F(G)}{|V(G)|^{|V(F)|}}$ denote the {\it\(F\)-density} of \(G\).
For two fixed graphs \(F\) and \(H\), define
\[
        \varphi_{F,H}(x)
        \coloneqq
        \inf \left\{
        \liminf_{n\to\infty} p_H(G_n):
        |V(G_n)|\to\infty,\ 
        \liminf_{n\to\infty} p_F(G_n)\ge x
        \right\},
\]
whenever the set on the right-hand side is non-empty. Thus \(\varphi_{F,H}(x)\) is the asymptotically smallest possible \(H\)-density among graphs whose \(F\)-density is at least \(x\). Evidently, \(\varphi_{F,H}\) is non-decreasing.

In the case of cliques, Theorem~\ref{THM:ordinary-main} determines this function explicitly. For \(2\le s<t\), it gives
\[
        \varphi_{K_s,K_t}(x)
        =
        F_t(F_s^{-1}(x)).
\]
The domain of this function is naturally divided into intervals by the critical points \(\theta_{p,s}\), and on each such interval the function is concave (see Figure~\ref{fig:Curve}). This suggests the following general problem.

\begin{problem}\label{PROB:general-pairs}
	Determine all pairs \((F,H)\) of graphs for which the function
	\(\varphi_{F,H}\) is piecewise concave.
\end{problem}

Here ``piecewise concave'' means that the domain of \(\varphi_{F,H}\) can be divided into finitely or countably many intervals such that the restriction of \(\varphi_{F,H}\) to each interval is concave. The above problem appears to be quite far-reaching. A natural first step is to consider the case \(F=K_2\), and to determine for which graphs \(H\) the edge-to-\(H\) function \(\varphi_{K_2,H}\) is piecewise concave.

\section*{Acknowledgments}
\noindent J.M. was supported by the National Key Research and Development Program of China (2023YFA1010201) and the National Natural Science Foundation of China grant 12125106.
T.W. and T.Z. were supported by the Innovation Program for Quantum Science and Technology (2021ZD0302902).

\section*{Declaration on the Use of AI} 
\noindent The main result of this paper, namely the sharp \(K_s\to K_t\) statement for all \(2\le s<t\), was conceived and formulated by the authors. During the exploratory stage of this work, AI-assisted tools were used to generate candidate proof ideas and assist with calculations. In particular, these tools helped the authors obtain an initial derivation of the \(K_3 \to K_4\) result on the first non-trivial interval \(1/27 \le p_3(G) \le 1/16\). 
Building on this initial derivation, the authors developed the mathematical framework of the paper and established the results presented herein. 
AI-assisted tools were also used for formula checking and proofreading.
All AI-assisted outputs were independently checked and verified by the authors, who take full responsibility for the content of this work.

\bibliographystyle{plain}
\bibliography{references}

@article{bollobas1976complete,
  author  = {Bollob{\'a}s, B{\'e}la},
  title   = {On complete subgraphs of different orders},
  journal = {Mathematical Proceedings of the Cambridge Philosophical Society},
  volume  = {79},
  number  = {1},
  pages   = {19--24},
  year    = {1976},
  doi     = {10.1017/S0305004100052063}
}

@article{erdos1986asymptotic,
  author  = {Erd{\H{o}}s, Paul and Frankl, P{\'e}ter and R{\"o}dl, Vojt{\v{e}}ch},
  title   = {The asymptotic number of graphs not containing a fixed subgraph
             and a problem for hypergraphs having no exponent},
  journal = {Graphs and Combinatorics},
  volume  = {2},
  pages   = {113--121},
  year    = {1986},
  doi     = {10.1007/BF01788085}
}

@article{fisher1989lower,
  author  = {Fisher, David Charles},
  title   = {Lower bounds on the number of triangles in a graph},
  journal = {Journal of Graph Theory},
  volume  = {13},
  number  = {4},
  pages   = {505--512},
  year    = {1989},
  doi     = {10.1002/jgt.3190130411}
}

@article{goodman1959sets,
  author  = {Goodman, Adolph Winkler},
  title   = {On sets of acquaintances and strangers at any party},
  journal = {The American Mathematical Monthly},
  volume  = {66},
  number  = {9},
  pages   = {778--783},
  year    = {1959},
  doi     = {10.1080/00029890.1959.11989408}
}

@article{kim2020asymptotic,
  author  = {Kim, Jaehoon and Liu, Hong and Pikhurko, Oleg and Sharifzadeh, Maryam},
  title   = {Asymptotic structure for the clique density theorem},
  journal = {Discrete Analysis},
  year    = {2020},
  pages   = {Paper No. 19, 26 pp.},
  doi     = {10.19086/da.18559}
}

@book{lovasz2012large,
  author    = {Lov{\'a}sz, L{\'a}szl{\'o}},
  title     = {Large Networks and Graph Limits},
  series    = {American Mathematical Society Colloquium Publications},
  volume    = {60},
  publisher = {American Mathematical Society},
  address   = {Providence, RI},
  year      = {2012},
  isbn      = {978-0-8218-9085-1}
}

@incollection{lovasz1983number,
  author    = {Lov{\'a}sz, L{\'a}szl{\'o} and Simonovits, Mikl{\'o}s},
  title     = {On the number of complete subgraphs of a graph. {II}},
  booktitle = {Studies in Pure Mathematics},
  pages     = {459--495},
  publisher = {Birkh{\"a}user},
  address   = {Basel},
  year      = {1983}
}

@article{moon1962problem,
  author  = {Moon, John W. and Moser, Leo},
  title   = {On a problem of {T}ur{\'a}n},
  journal = {Magyar Tud. Akad. Mat. Kutat{\'o} Int. K{\"o}zl.},
  volume  = {7},
  pages   = {283--286},
  year    = {1962}
}

@article{nikiforov2011number,
  author  = {Nikiforov, Vladimir},
  title   = {The number of cliques in graphs of given order and size},
  journal = {Transactions of the American Mathematical Society},
  volume  = {363},
  number  = {3},
  pages   = {1599--1618},
  year    = {2011},
  doi     = {10.1090/S0002-9947-2010-05189-X}
}

@article{nordhaus1963triangles,
  author  = {Nordhaus, Edward A. and Stewart, Bonnie Madison},
  title   = {Triangles in an ordinary graph},
  journal = {Canadian Journal of Mathematics},
  volume  = {15},
  pages   = {33--41},
  year    = {1963},
  doi     = {10.4153/CJM-1963-004-7}
}

@article{pikhurko2017asymptotic,
  author  = {Pikhurko, Oleg and Razborov, Alexander A.},
  title   = {Asymptotic structure of graphs with the minimum number of triangles},
  journal = {Combinatorics, Probability and Computing},
  volume  = {26},
  number  = {1},
  pages   = {138--160},
  year    = {2017},
  doi     = {10.1017/S0963548316000110}
}

@article{razborov2007flag,
  author  = {Razborov, Alexander A.},
  title   = {Flag algebras},
  journal = {Journal of Symbolic Logic},
  volume  = {72},
  number  = {4},
  pages   = {1239--1282},
  year    = {2007},
  doi     = {10.2178/jsl/1203350785}
}

@article{razborov2008minimal,
  author  = {Razborov, Alexander A.},
  title   = {On the minimal density of triangles in graphs},
  journal = {Combinatorics, Probability and Computing},
  volume  = {17},
  number  = {4},
  pages   = {603--618},
  year    = {2008},
  doi     = {10.1017/S0963548308009085}
}

@article{reiher2016clique,
  author  = {Reiher, Christian},
  title   = {The clique density theorem},
  journal = {Annals of Mathematics},
  series  = {2},
  volume  = {184},
  number  = {3},
  pages   = {683--707},
  year    = {2016},
  doi     = {10.4007/annals.2016.184.3.1}
}

@article{turan1941extremal,
  author  = {Tur{\'a}n, P{\'a}l},
  title   = {Eine Extremalaufgabe aus der Graphentheorie},
  journal = {Matematikai {\'e}s Fizikai Lapok},
  volume  = {48},
  pages   = {436--452},
  year    = {1941}
}

@article{erdos1962number,
    author = {Paul Erd\H{o}s},
    title = {On the Number of Complete Subgraphs Contained in Certain Graphs},
    journal = {Magyar Tud. Akad. Mat. Kutat{\'o} Int. K{\"o}zl.},
    volume  = {7},
    number = {3},
    pages   = {459--464},
    year ={1962} 
}

@article{liu2020exact,
  author  = {Liu, Hong and Pikhurko, Oleg and Staden, Katherine},
  title   = {The exact minimum number of triangles in graphs with given order and size},
  journal = {Forum of Mathematics, Pi},
  volume  = {8},
  pages   = {e8},
  year    = {2020},
  doi     = {10.1017/fmp.2020.7}
}

@article{alon2016,
  author  = {Alon, Noga and Shikhelman, Clara},
  title   = {Many {T} copies in {H}-free graphs},
  journal = {Journal of Combinatorial Theory, Series B},
  volume  = {121},
  pages   = {146--172},
  year    = {2016},
  doi     = {10.1016/j.jctb.2016.03.004}
}

@article{Khadzhiivanov1978,
  author  = {Khad\v{z}iivanov, Nikolai G. and Nikiforov, Vladimir S.},
  title   = {The {Nordhaus--Stewart--Moon--Moser} inequality},
  journal = {Serdica},
  volume  = {4},
  number  = {4},
  pages   = {344--350},
  year    = {1978},
  note    = {In Russian}
}

@article{zykov1949linear,
  author   = {Zykov, Aleksandr A.},
  title    = {On some properties of linear complexes},
  journal  = {Matematicheskii Sbornik. Novaya Seriya},
  volume   = {24(66)},
  number   = {2},
  pages    = {163--188},
  year     = {1949},
  language = {Russian}
}

\appendix

\section{Proof of Lemma~\ref{LEM:quadratic-MM}}\label{APP:quadratic}
\noindent We include a short proof of Lemma~\ref{LEM:quadratic-MM}, adapted from the proof of Proposition 3.1 in Reiher~\cite{reiher2016clique}. 
The argument may be viewed as a weighted double-counting of the pairs $(L,\{u,v\})$, where $L\subseteq V(G)$ has size $r-1$ and $u,v\in V(G)\setminus L$ (not necessarily distinct) are such that both $L\cup \{u\}$ and $L\cup \{v\}$ induce copies of $K_r$ in $G$.

\begin{proof}[\bf Proof of Lemma~\ref{LEM:quadratic-MM}]
Let \(\GG=\GG(\vec{\mathbf{x}}, \bA)\) be a weighted graph of order \(n\).  For \(S\subseteq[n]\) write
\[
        A_S\coloneqq\prod_{e\in \binom{S}{2}}a_e
        \quad\text{and}\quad
        X_S\coloneqq\prod_{i\in S}x_i.
\]
Fix an integer \(r\ge2\). For each \(M\in \binom{[n]}{r+1}\), define
\[
        B_M\coloneqq\sum_{e\in \binom{M}{2}}\prod_{f\in \binom{M}{2}\setminus\{e\}}a_f,
        \quad \text{and}\quad
        C_M\coloneqq\sum_{N\in \binom{M}{r}}A_N.
\]
We first claim that
\begin{equation}\label{eq:appendix-B-C}
        2B_M-C_M\le (r^2-1)A_M.
\end{equation}
Both sides are multilinear in the variables \(a_e\) corresponding to the pairs \(e\subseteq M\). Therefore, it suffices to verify \eqref{eq:appendix-B-C} when each \(a_e\in\{0,1\}\).
If at least two pairs in \(M\) have weight \(0\), then \(A_M=B_M=0\) and \eqref{eq:appendix-B-C} is immediate. If exactly one pair has weight \(0\), then \(A_M=0\) and \(B_M=1\). Moreover, the only \(r\)-subsets that form cliques are the two obtained by deleting one of the vertices of that pair; hence \(C_M=2\).
Finally, if no pair has weight \(0\), then \(A_M=1\), \(B_M=\binom{r+1}{2}\), and \(C_M=r+1\). Therefore,
\(2B_M-C_M=2\binom{r+1}{2}-(r+1)=r^2-1.\)
This proves \eqref{eq:appendix-B-C}. 
 
Multiplying \eqref{eq:appendix-B-C} by \(X_M\) and
summing over all \(M\in\binom{[n]}{r+1}\), we get
\begin{equation}\label{eq:appendix-sumBC}
        \sum_{M\in\binom{[n]}{r+1}}(2B_M-C_M)X_M
        \le (r^2-1)\GG(K_{r+1}).
\end{equation}

For each \(L\in\binom{[n]}{r-1}\), we define
\[
        \eta_L\coloneqq\sum_{i\in[n]\setminus L} \left(x_i\prod_{\ell\in L}a_{i\ell}\right)
        \quad \text{and} \quad
        \Omega\coloneqq\sum_{L\in\binom{[n]}{r-1}}A_LX_L\eta_L^2.
\]
Expanding \(\eta_L^2\), we write $\Omega=\Omega_{1}+\Omega_{2}$, where  
\begin{align*}
\Omega_{1}
&\coloneqq \sum_{L\in \binom{[n]}{r-1}} A_LX_L \sum_{i\neq j, \{i,j\}\cap L=\emptyset} \left(x_i\prod_{\ell\in L}a_{i\ell}\right) \left(x_j\prod_{\ell\in L}a_{j\ell}\right)= 2\sum_{M\in\binom{[n]}{r+1}}B_MX_M,\\
\Omega_{2}
&\coloneqq\sum_{L\in\binom{[n]}{r-1}}A_LX_L
  \sum_{i\in[n]\setminus L}x_i^2\prod_{\ell\in L}a_{i\ell}^2  \le
  \sum_{L\in\binom{[n]}{r-1}}A_LX_L
  \sum_{i\in[n]\setminus L}x_i^2\prod_{\ell\in L}a_{i\ell} \\
&=\sum_{Q\in\binom{[n]}{r}}A_QX_Q\sum_{i\in Q}x_i
 =\sum_{Q\in\binom{[n]}{r}}A_QX_Q\left(1-\sum_{i\notin Q}x_i\right)
 =\GG(K_r)-\sum_{M\in\binom{[n]}{r+1}}C_MX_M.
\end{align*}
Then using \eqref{eq:appendix-sumBC}, we have
\[
        \Omega
        \le \GG(K_r)+\sum_{M\in\binom{[n]}{r+1}}(2B_M-C_M)X_M
        \le \GG(K_r)+(r^2-1)\GG(K_{r+1}).
\]
Finally, by Cauchy's inequality, we derive 
\[
\big(r\GG(K_r)\big)^2=\left(\sum_{L\in\binom{[n]}{r-1}}A_LX_L\eta_L\right)^2
\le
\left(\sum_{L\in\binom{[n]}{r-1}}A_LX_L\right)
\left(\sum_{L\in\binom{[n]}{r-1}}A_LX_L\eta_L^2\right)=\GG(K_{r-1})\cdot \Omega.
\]
Combining the two inequalities above, we obtain
\[
r^2\GG(K_r)^2\le \GG(K_{r-1})\cdot \Omega\le \GG(K_{r-1})\bigl(\GG(K_r)+(r^2-1)\GG(K_{r+1})\bigr).
\]
This proves the lemma. 
\end{proof}

\section{Supplementary Calculations}\label{APP:derivatives}

\noindent
This appendix supplies the detailed calculations that are stated without proof in the main text. 

\vspace{1em}

\noindent
\textbf{\underline{Verification for Equation~\eqref{eq:log-derivative}.}} 
Let $\widetilde{r}\in \{r-1,r\}$. From the definition of \(F_{\widetilde{r}}(\gamma)\), taking logarithms leads to
  \begin{align*}
      \log F_{\widetilde{r}}(\gamma)
      &= ({\widetilde{r}}-1)\log(1+\alpha) + \log\bigl(1-({\widetilde{r}}-1)\alpha\bigr) + f(p,{\widetilde{r}}),
  \end{align*}
where \(f(p,{\widetilde{r}})\) denotes a function only depending on $p$ and ${\widetilde{r}}$.  
Since $0<\alpha<1/p\leq 1/(r-1)$, we obtain
\begin{align*}
     \frac{\dd}{\dd\alpha}
      \log\left(\frac{F_r(\gamma)}{F_{r-1}(\gamma)^{r/(r-1)}}\right)=
      &\frac{r-1}{1+\alpha} - \frac{r-1}{1-(r-1)\alpha}
      - \frac{r}{r-1}\Bigl(\frac{r-2}{1+\alpha} - \frac{r-2}{1-(r-2)\alpha}\Bigr)  \\
      =&\ \frac{1}{1+\alpha}\cdot\frac{1}{r-1}
        - \frac{r-1}{1-(r-1)\alpha}
        + \frac{r(r-2)}{(r-1)(1-(r-2)\alpha)} \\
      =&-\frac{r\alpha}{(1+\alpha)(1-(r-2)\alpha)(1-(r-1)\alpha)}<0.
\end{align*}
This verifies Equation \eqref{eq:log-derivative}.\qed

\vspace{1em}

\noindent
\textbf{\underline{Verification for Equation~\eqref{eq:slope}.}} 
Let $\widetilde{q}\in \{q-1,q\}$.
Recall that \(Z_{\widetilde{q}}(x)\) is defined by writing \(P_{r+1}(\widetilde{q},\beta)\) as a function of \(x = P_r(\widetilde{q},\beta)\). 
First, by the definition of \(P_{r}(\widetilde{q},\beta)\), we have 
  \begin{align*}
      \frac{\dd}{\dd\beta}P_r(\widetilde{q},\beta)
      = \frac{\binom{\widetilde{q}}{r}}{(\widetilde{q})^r}
         \Bigl[(r-1)(1+\beta)^{r-2}\bigl(1-(r-1)\beta\bigr)
               + (1+\beta)^{r-1}\bigl(-(r-1)\bigr)\Bigr] 
      = -\frac{\binom{\widetilde{q}}{r}}{(\widetilde{q})^r}
         \cdot r(r-1)\beta(1+\beta)^{r-2}.
  \end{align*}
Similarly, we obtain
\(\frac{\dd}{\dd\beta}P_{r+1}(\widetilde{q},\beta) = -\frac{\binom{\widetilde{q}}{r+1}}{(\widetilde{q})^{r+1}}
         \cdot (r+1)r\beta(1+\beta)^{r-1}.\)
Then using the chain rule, we have
  \[
      Z_{\widetilde{q}}'(x)
      =\frac{\dd P_{r+1}(\widetilde{q},\beta)}
             {\dd P_r(\widetilde{q},\beta)} =\frac{\dd P_{r+1}(\widetilde{q},\beta)/\dd\beta}
             {\dd P_r(\widetilde{q},\beta)/\dd\beta}
      = \frac{\binom{\widetilde{q}}{r+1}}{\binom{\widetilde{q}}{r}}\cdot\frac{1}{\widetilde{q}}\cdot\frac{r+1}{r-1}\cdot(1+\beta)
      = \frac{\widetilde{q}-r}{\widetilde{q}(r-1)}(1+\beta),
  \]
  as claimed in Equation~\eqref{eq:slope}.
\qed

\vspace{1em}

\noindent
\textbf{\underline{Verification for Equation~\eqref{eq:value-of-lambda}.}} 
By the chain rule and the formula for the derivative of an inverse function,
\begin{equation*}
      \lambda=\Phi_r'(x_0)
      = F_{r+1}'\bigl(F_r^{-1}(x_0)\bigr)\cdot (F_r^{-1})'(x_0)
      = \frac{F_{r+1}'(\gamma)}{F_r'(\gamma)} = \frac{\dd F_{r+1}/\dd\alpha}{\dd\gamma/\dd\alpha}\big/ \frac{\dd F_r/\dd\alpha}{\dd\gamma/\dd\alpha}=\frac{\dd F_{r+1}/\dd \alpha}{\dd F_{r}/\dd\alpha}.
\end{equation*}
For $\widetilde{r}\in \{r,r+1\}$,  it is straightforward to compute the derivative 
\(\frac{\dd}{\dd\alpha}\Bigl[(1+\alpha)^{\widetilde{r}-1}\bigl(1-(\widetilde{r}-1)\alpha\bigr)\Bigr]\) 
as
\begin{align*}
   (\widetilde{r}-1)(1+\alpha)^{\widetilde{r}-2}\bigl(1-(\widetilde{r}-1)\alpha\bigr)
         + (1+\alpha)^{\widetilde{r}-1}\bigl(-(\widetilde{r}-1)\bigr) = -\widetilde{r}(\widetilde{r}-1)\alpha(1+\alpha)^{\widetilde{r}-2}, 
\end{align*}
which implies that
\[
      \frac{\dd F_{\widetilde{r}}}{\dd \alpha}= \frac{\binom{p+1}{\widetilde{r}}}{(p+1)^{\widetilde{r}}}
      \cdot \frac{\dd}{\dd \alpha}\Bigl[(1+\alpha)^{\widetilde{r}-1}\bigl(1-(\widetilde{r}-1)\alpha\bigr)\Bigr]
      = -\frac{\binom{p+1}{\widetilde{r}}}{(p+1)^{\widetilde{r}}}\cdot \widetilde{r}(\widetilde{r}-1)\alpha(1+\alpha)^{\widetilde{r}-2}.
\]
Using the identity \(\binom{p+1}{r+1}/\binom{p+1}{r} = (p+1-r)/(r+1)\), we have
\begin{align*}
      \lambda 
      = \frac{\dd F_{r+1}/\dd\alpha}{\dd F_r/\dd\alpha} 
      = \frac{\binom{p+1}{r+1}}{\binom{p+1}{r}}\cdot\frac{1}{p+1}\cdot\frac{r+1}{r-1}\cdot(1+\alpha)
      = \frac{p-r+1}{(r-1)(p+1)}(1+\alpha).
\end{align*}
This completes the calculation of \(\lambda\) and verifies Equation~\eqref{eq:value-of-lambda}.
\qed

\vspace{1em}

\noindent
\textbf{\underline{Verification for Equation~\eqref{eq:key-eq-clique-lifting}.}}
Fix \(i\in [n]\). Throughout this proof, we let 
\[
L_{i}\coloneqq \frac{\eta_{i}+r(r-2)}{r(r-1)}. 
\]
It follows directly from the definition of \(\beta = \frac{r-\eta_i}{\eta_i+r(r-2)}\) and \(L_{i}\) that 
\begin{equation}\label{eq:app-beta-L}
      1+\beta = \frac{1}{L_i}, \quad
      1-(r-2)\beta = \frac{\eta_i}{rL_i} \quad\text{and}\quad 
      1-(r-1)\beta = \frac{\eta_i-1}{(r-1)L_i}. 
  \end{equation}
Before we verify the desired identities, let us recall all the relevant definitions and formulas as follows:
  \[
      \GG_i(K_{r-1}) = \frac{\mu\eta_i}{\lambda}, \quad \GG_i(K_r) = \mu(\eta_i-1),\quad 
      \tau_i = \frac{p(1+\alpha)}{p+1}M^{1/(r-1)}\cdot \frac{\eta_{i}+r(r-2)}{r(r-1)}=\frac{p(1+\alpha)}{p+1}M^{1/(r-1)}L_i,
  \]
  \[
      \mu = M\mu_0,\quad\mu_0 = \frac{r+1}{r-1}\frac{\binom{p+1}{r+1}}{(p+1)^{r+1}}(1+\alpha)^r,
      \quad \mbox{and} \quad
      \lambda = \frac{p-r+1}{(r-1)(p+1)}(1+\alpha). 
  \]
  
  Now we verify the first identity.
  We first compute the ratio \(\mu_0/\lambda\) as
  \[
  \frac{\mu_0}{\lambda}
  = \frac{\frac{r+1}{r-1}\frac{\binom{p+1}{r+1}}{(p+1)^{r+1}}(1+\alpha)^r}
         {\frac{p-r+1}{(r-1)(p+1)}(1+\alpha)}
  = \frac{r+1}{p-r+1}\cdot\frac{\binom{p+1}{r+1}}{(p+1)^r}\cdot(1+\alpha)^{r-1}
  = \frac{\binom{p+1}{r}}{(p+1)^r}(1+\alpha)^{r-1}.
  \]
  Using~\eqref{eq:app-beta-L}, we obtain
  \[
  \frac{\eta_i}{rL_i^{\,r-1}}
  = \frac{\eta_i}{rL_i}\cdot\frac{1}{L_i^{\,r-2}}
  = \bigl(1-(r-2)\beta\bigr)(1+\beta)^{r-2}.
  \]
  Using the above expressions, we derive the first identity as follows:
  \begin{align*}
  \frac{\GG_i(K_{r-1})}{\tau_i^{r-1}}
  &= \frac{\mu\eta_i/\lambda}
         {\bigl(\frac{p(1+\alpha)}{p+1}M^{1/(r-1)}L_i\bigr)^{r-1}}
  = \frac{\mu_0}{\lambda}\cdot \frac{\eta_i}{rL_i^{r-1}}\cdot
    \frac{r(p+1)^{r-1}}{p^{r-1}(1+\alpha)^{r-1}}\\
  &= \frac{\binom{p+1}{r}}{(p+1)^r}(1+\alpha)^{r-1}\cdot \bigl(1-(r-2)\beta\bigr)(1+\beta)^{r-2}\cdot \frac{r(p+1)^{r-1}}{p^{r-1}(1+\alpha)^{r-1}}\\
  &= \frac{\binom{p}{r-1}}{p^{r-1}}
     (1+\beta)^{r-2}\bigl(1-(r-2)\beta\bigr)
  = P_{r-1}(p,\beta),
  \end{align*}
  where the second equation follows by $\mu=M\mu_0$ and the second last equation holds because \(r\binom{p+1}{r}=(p+1)\binom{p}{r-1}\). 

  Next we verify the second identity by similar computations.
  Using~\eqref{eq:app-beta-L}, \(\eta_i-1 = (r-1)L_i\bigl(1-(r-1)\beta\bigr)\), so
  \[
  \frac{\eta_i-1}{L_i^{\,r}}
  = \frac{(r-1)(1-(r-1)\beta)}{L_i^{\,r-1}}
  = (r-1)(1-(r-1)\beta)(1+\beta)^{r-1}.
  \]
  This, together with \(\GG_i(K_r) = \mu(\eta_i-1)=M\mu_0(\eta_i-1)\), implies that
  \begin{align*}
  \frac{\GG_i(K_r)}{\tau_i^r}
  &= \frac{M\mu_0(\eta_i-1)}
         {\bigl(\frac{p(1+\alpha)}{p+1}M^{1/(r-1)}L_i\bigr)^r}
  = M^{-1/(r-1)}\cdot\mu_0\cdot \frac{\eta_i-1}{L_i^r}\cdot
    \frac{(p+1)^r}{p^r(1+\alpha)^r}\\
  &= M^{-1/(r-1)}\cdot\frac{r+1}{r-1}\frac{\binom{p+1}{r+1}}{(p+1)^{r+1}}(1+\alpha)^r
    \cdot (r-1)(1-(r-1)\beta)(1+\beta)^{r-1} \cdot \frac{(p+1)^r}{p^r(1+\alpha)^r}\\
  & = M^{-1/(r-1)}\cdot\frac{\binom{p}{r}}{p^r}
      (1+\beta)^{r-1} (1-(r-1)\beta)
      = M^{-1/(r-1)}\cdot P_r(p,\beta),
  \end{align*}
 where the second last equation holds because \((r+1)\binom{p+1}{r+1}=(p+1)\binom{p}{r}\).
 This proves the second identity and completes the verification.
\qed

\end{document}